\documentclass{article}

\usepackage[T1]{fontenc}
\usepackage[utf8]{inputenc}
\usepackage{lmodern}
\usepackage{microtype}

\usepackage[total={5.8in, 7.8in}, asymmetric, bindingoffset=0.4in]{geometry}
\usepackage{amsmath}
\usepackage{amssymb}
\usepackage{amsthm}
\usepackage{mathtools}
\usepackage{bbm}
\usepackage{enumerate}
\usepackage{enumitem}
\usepackage[dvipsnames]{xcolor}

\usepackage{subcaption}

\usepackage{multirow}
\usepackage{empheq}
\usepackage{listings}
\usepackage{color}

\usepackage[square,comma,numbers]{natbib}

\usepackage{graphicx} %
\usepackage[font=small,labelfont=bf]{caption}
\usepackage{tikz}
\usetikzlibrary{babel}

\numberwithin{equation}{section}
\usepackage{pgfplots}
\pgfplotsset{width=10cm,compat=1.9}

\usepackage[colorlinks=true, urlcolor=blue, linkcolor=blue]{hyperref}

\newcommand*\indic[1]{\mathbbm{1}_{\{ #1 \}}}
\newcommand*\indica[1]{\mathbbm{1}_{ #1 }}

\def\cL{{\mathcal L}}

\def\cP{{\mathcal P}}

\def\C{\mathbb{C}}

\def\E{\mathbb{E}}

\def\P{\mathbb{P}}
\def\R{\mathbb{R}}

\def\d{\mathrm{d}}

\providecommand{\abs}[1]{\lvert#1\rvert}

\usepackage{MnSymbol}
\theoremstyle{plain}
\newtheorem{theorem}{Theorem}[section]
\newtheorem{lemma}[theorem]{Lemma}

\newtheorem{proposition}[theorem]{Proposition}
\newtheorem{assumption}[theorem]{Assumption}

\newtheorem{remark}[theorem]{Remark}

\usepackage[symbol]{footmisc}

\usepackage{fancyhdr}
\newcommand{\dontprintdoi}[1]{}
\AtBeginDocument{\let\doi\dontprintdoi}
\makeatother

\date{\vspace{-1em}\normalsize{\today}}

\title{On large networks of integrate-and-fire neurons with short-term synaptic plasticity
}

\author{Quentin Cormier\footnote{Inria Saclay, CMAP, Ecole Polytechnique, IP Paris, France, email:
{\tt quentin.cormier@inria.fr}} , Eva L\"{o}cherbach\footnote{CMAP, Ecole Polytechnique, IP Paris, France, email: {\tt eva.loecherbach@polytechnique.edu}} and Valentin Schmutz\footnote{Mathematical Institute, University of Oxford, United Kingdom, email: {\tt valentin.schmutz@maths.ox.ac.uk}}}

\date{\today}
\begin{document}
\maketitle

\begin{abstract}
    This work studies the mean-field limit of large networks of interacting
    stochastic leaky integrate-and-fire (LIF) neurons subject to short-term synaptic depression (STD). The macroscopic dynamics of this system is governed by a two-dimensional, non-linear McKean-Vlasov equation that couples the evolution of the neurons' membrane potentials with a synaptic depression variable. We investigate the long-time behavior of this limit system. To this end, we introduce an auxiliary linearized Markov process by freezing the interaction non-linearity to a constant. By exploiting the regeneration of the membrane potential at spike times, we are able to explicitly compute the conditional expectation of the synaptic depression variable, conditionally on the potential value, under the invariant measure of this two-dimensional linear process. This is a crucial ingredient to study time-dependent local perturbations thereof. As a consequence we are able to identify an analytic criterion guaranteeing the local stability of any invariant probability measure of the fully non-linear system. This stability criterion is formulated in terms of the zeros of the Laplace transform of a specific linear response function. Finally, we provide numerical examples demonstrating that the two-dimensional framework induces a richer spectrum of long-time dynamics than purely one-dimensional models. For example, synaptic depression can lead to low-frequency oscillations around a unique, unstable invariant measure where the oscillations are much slower than the neurons' firing rates. 
\end{abstract}
\vspace{10pt}
{\small
}
\vspace{10pt}

\vspace{5pt}
\section{Introduction}
\subsection{Spiking neurons with short term synaptic depression}\label{sec:model}
We consider systems of interacting spiking neurons with short-term synaptic depression (STD), in their mean-field limits. The finite system version of our model is made of $n$ spiking neurons. Each neuron $i$ is characterized by its membrane potential $ V_t^{n, i } $ and its synaptic depression variable $ X_t^{n, i} $ taking values in $ [0, 1]$. The neuron spikes randomly, at rate $ f (V_{t-}^{n, i })$,  depending only on its membrane potential, independently of the other neurons. At any spiking time $T$ in the system, the following happens. If it is neuron $i$ that is spiking, then the membrane potential of neuron $i$ is reset to the resting value $0$ such that we have $ V_T^{n, i } := 0$, and its synaptic depression variable $X_{T-}^{n, i } $ decreases by $ U X_{T-}^{n, i } $ and is replaced by the new value $ X_{T}^{n, i }:= (1 - U ) X_{T-}^{n, i }.$ At the same time, all other neurons $ j \neq i $  have their membrane potentials changed by the additional amount  $ J  X_{T-}^{n, i }/n, $ where $ J \in \R $ is a fixed constant, that is, for all $j \neq i , $ 
$$V_T^{n, j } := V_{T-}^{n, j } + \frac{J X_{T-}^{n, i }}{ n}. $$
In between successive spikes in the system, each neuron's potential follows a deterministic evolution according to the ordinary differential equation (ODE) $ \d V_t^{n, i } = b( V_t^{n, i }) \d t,$ where $ b : \R \to \R$ describes the subthreshold dynamics of each neuron. Finally, the depression variables recover at exponential rate, that is, for some fixed constant $ \tau > 0, $ $ \tau \d X_t^{n, i } = ( 1 - X_t^{n, i } ) \d t .$ The variables $ X_t^{n, i }$ act as a fatigue mechanism on the synaptic transmission process and they are typically called the synaptic depression variables. We can think of $ X_t^{n, i } $ as representing, e.g., the number of synaptic vesicles that are ready for release in the axon terminal of neuron $i.$%

\subsection{Relation to previous models in computational neuroscience}
The model of short-term synaptic depression (STD) we use is the phenomenological model of Tsodyks and Markram developed in \cite{TsodyksMarkram1997} and simplified in \cite{Tsodyks1998}. STD can have important effects on information transmission between neurons \cite{Abbott1997,Pfister2010,Rosenbaum2012} and it is involved in a recently proposed model of biologically plausible learning in multilayer spiking neural networks \cite{Payeur2021}.

While networks of LIF neurons with STD have been analysed in several simulation studies \cite{Tsodyks2000,Loebel2002,Mongillo2005,Schmutz2020}, very few works have attempted a theoretical analysis of this type of model. In \cite{Romani2006}, the authors propose a heuristic mean-field analysis of the stationary state of networks of LIF neurons with STD and derive a formula expressing the mean depression variable in terms of the stationary interspike interval distribution. Considering a different stochastic spiking neuron model, namely age-dependent nonlinear Hawkes processes \cite{Chevallier2017} instead of LIF neurons with escape noise \cite{Gerstner2000,GalvesLocherbach2016}, a multidimensional McKean--Vlasov equation has been rigorously derived from systems of interacting neurons with Tsodyks--Markram short-term synaptic plasticity in \cite{Schmutz2022}. The corresponding two-dimensional limit equation for depressing-only synapses has been studied in \cite{FonteSchmutz2022}, where a closed-form expression for the stationary firing rate of the linear equation was derived. Although age-dependent nonlinear Hawkes processes and LIF neurons are different models, they are structurally closely related \cite{Gerstner1995}. 

If the membrane potential is not reset after each spike, the mean-field equation becomes much simpler. This was already observed in \cite{Tsodyks1998} and later proved rigorously in \cite{ACEE} in the case of purely facilitating synapses. Finite-size fluctuations for this simpler model have been studied, through non-rigorous methods, in \cite{Schmutz2020,Pietras2022}. For the analysis of the long-time behavior of two-dimensional mean-field models of networks of LIF neurons with spike-triggered adaptation, a fatigue mechanism different from STD, we refer the reader to \cite{salort2024,Veltz2025,Ambrogi2026}.

\subsection{Long-time behavior of the associated mean-field limit}
Let us come back to the model introduced in Section~\ref{sec:model} above. 
As the number of neurons $n$ tends to infinity, under suitable assumptions on the parameters of the model, the above system converges to its associated mean-field limit. In this mean-field limit, neurons become independent. Therefore, to describe the limit system, it is sufficient to describe the typical evolution of a single fixed neuron having potential value $ \bar V_t$ and a synaptic depression variable $\bar X_t.$ This limit dynamics is given by the non-linear equation of McKean-Vlasov type
\begin{eqnarray}\label{eq:VXNLintro}
	\d \bar V^\mu_t & =& b(\bar V^\mu_t)\d t + J \E [ \bar X^\mu_t f( \bar V^\mu_t)] \d t - \bar V^\mu_{t-} \int_{ \R_+} \indic{z \leq f(\bar V^\mu_{t-})} N(\d t, \d z) , \nonumber \\
	\d \bar X^\mu_t &=& \frac{1 - \bar X^\mu_t}{\tau} \d t - U \bar X^\mu_{t-} \int_{ \R_+} \indic{z \leq f(\bar V^\mu_{t-})} N(\d t, \d z) ,
\end{eqnarray}
where $ N ( \d t , \d z )$ is a Poisson random measure on $ \R_+ \times \R_+ $ having intensity $ \d t \d z,$ and where $ \mu $ is the initial law of the process, that is, $ ( \bar V_0^\mu, \bar X_0^\mu)\sim \mu .$ 

We work under standard regularity assumptions on the coefficients (see Assumption \ref{ass:1} below for the precise statement) and suppose in particular that $f$ and $b$ are Lipschitz continuous and that $f$ is bounded. In this framework, it is standard to prove the convergence of the finite system to its mean-field limit; we will not treat this point here and we refer to \cite{eva-fou,CTV,ACEE,Schmutz2022} for such studies. In the present paper we restrict our attention to the mere study of the mean-field limit, with a focus on its long-time behavior. This is not a trivial task since the evolution in (\ref{eq:VXNLintro}) is not Markovian. Indeed, due to the presence of the interaction term $J \E ( \bar X^\mu_t f( \bar V^\mu_t))  $ in the limit drift, the limit equation is non-linear and depends on the law of the process. Describing invariant measures of this dynamics, their attractiveness and structure is therefore in general  difficult. While the finite system is Markovian and possesses in general at most one single invariant measure, the limit dynamics (\ref{eq:VXNLintro}) may have several invariant measures, some of them being attractive and some  not. In addition, even when the invariant measure is unique, it can be non-attractive and oscillations may appear (see the example in Section~\ref{eq:bistable} below). 

Notice that if we formally take $ U= 0 $ and $ \tau = +\infty,$ then $ \bar X^\mu_t = \bar X^\mu_0 $ for all $t \geq 0, $ such that our model is effectively one-dimensional. In this case, in \cite{CTV}, \cite{MNA} and \cite{cormier2023stability}, the long-time behavior of $ (\bar V_t^\mu)_{t \geq 0}$ has been extensively studied. In particular, it has been shown that both stable and oscillatory behaviors are possible, depending on the model parameters. In this paper, we are interested in the general two-dimensional model. We will show that the two-dimensional structure induces a richer spectrum of possible long-time behaviors than in the simpler one-dimensional case. To do so, we will adapt the approach of \cite{cormier2023stability} to the present setting. 

An important ingredient of this approach is the study of an auxiliary Markov process, the {\it linearized version} of (\ref{eq:VXNLintro}), which is obtained by freezing the non-linearity $J \E ( \bar X^\mu_t f( \bar V^\mu_t))  $ and replacing it by a constant $ \alpha. $ This gives rise to the linearized process $(V^{\alpha  }_t, X^{\alpha  }_t)$ which is the solution of the linear  equation
\begin{eqnarray}\label{eq:valphamuintro}
	\d V^{\alpha }_t & =& b( V^{\alpha}_t)\d t + \alpha  \d t - V^{\alpha}_{t-} \int_{ \R_+} \indic{z \leq f(V^{\alpha}_{t-})} N(\d t, \d z) , \nonumber \\
	\d X^{\alpha }_t &=& \frac{1 - X^{\alpha }_t}{\tau} \d t - U X^{\alpha }_{t-} \int_{ \R_+} \indic{z \leq f(V^{\alpha }_{t-})} N(\d t, \d z) . 
\end{eqnarray}
For any fixed constant $\alpha, $ this defines a two-dimensional Markov process with a first component exhibiting a regenerative structure induced by the reset to $0$ at each spiking time. This Markov process is Harris recurrent converging at exponential speed to its unique invariant probability measure. However, and this is one of the difficulties of not working in one-dimensional state space, the form of the two-dimensional invariant measure is not explicitly known (see Figure~\ref{fig:shape_invariant_distribution} for its typical shape);
\begin{figure}[ht]
\centering
\includegraphics[scale=0.7]{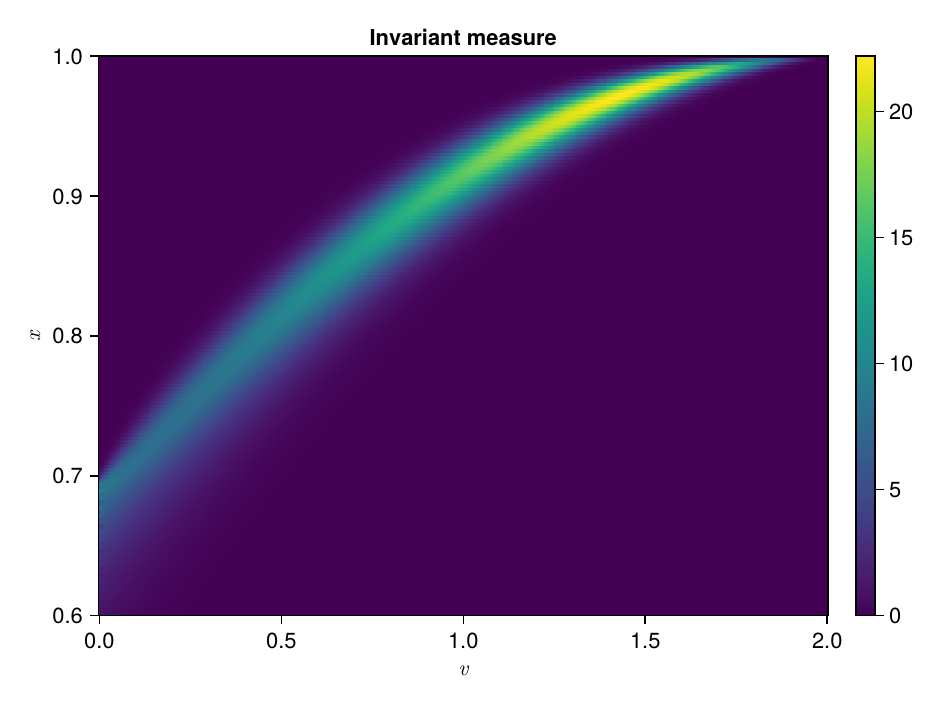}
\caption{The shape of the invariant measure $\mu_\infty(\d v, \d x)$ of $(V^\alpha_t, X^\alpha_t)$. Parameters: $b(v) = 0, \alpha=1$, $\tau=0.5$, $U = 0.3$, $f(v) = \max(0,v)^2$. }
\label{fig:shape_invariant_distribution}
\end{figure}
we only have an explicit expression of the first marginal, due to the regenerative structure of $V^\alpha_t.$ Despite this fact, the specific structure of the interactions and the linearity of the stochastic dynamics of $X^\alpha_t$ enable us to explicitly calculate the limit $ \lim_{ t \to \infty } \E ( X^\alpha_t|V^\alpha_t = v ) .$ This can be done for any fixed value of the constant $ \alpha,$ and it is a crucial step for our analysis. We refer to Proposition~\ref{prop:invm} below for the details.  

Then, we follow the approach of \cite{cormier2023stability} and study local perturbations around any possible invariant state of (\ref{eq:VXNLintro}) by studying the auxiliary Markov process perturbed by an input signal $ \alpha + a_t, $ depending on time. Here, $ a_t $ has to be thought of as a small perturbation. We are able to establish local stability results with respect to such perturbations. To do so, we rely on the {\it bounded Lipschitz} distance for probability measures on $ \R \times [0, 1 ],$ and we study Volterra equations that naturally appear when conditioning with respect to the first jump (spike). 

Our main result, Theorem~\ref{theo:2main}, gives an explicit criterion that enables us to decide if a given invariant probability measure $\mu_\infty(\d v, \d x)$ of (\ref{eq:VXNLintro}) is locally stable or not. This criterion is expressed in terms of the function 
\begin{equation}\label{eq:thetaintro}
\Theta_\alpha(t) := J \int_{\R \times [0,1]} \frac{\d }{\d v} \E_{(v,x)}[X^{\alpha}_t f(V^\alpha_t)] \mu_\infty(\d v, \d x) ,
\end{equation} 
where $ \alpha $ is chosen such that  $ \alpha = J \int_{\R \times [0,1]}  x f( v)  \mu_\infty(\d v, \d x) ,$ and of the zeros of its associated Laplace transform $ \hat \Theta_\alpha(z) = \int_0^\infty e^{ - z t } \Theta_\alpha(t) \d t.$ 

We stress that we are able to calculate $\Theta_\alpha(t)$  explicitly such that we can study concrete examples numerically (we refer to Section~\ref{sec:num} below). In particular we discuss in Section~\ref{sec:5} the following example: choosing
\begin{equation}\label{eq:bistable} b ( v) = 0.05 - v,  f (v) = \max(0,v)^2  , J = 6, U = 0.3 ,  \tau = 30,
\end{equation}
we observe an interesting oscillatory behavior, see Figure~\ref{fig:raster-SOC}.
\begin{figure}[!ht]
    \centering
    \includegraphics[width=0.8\textwidth]{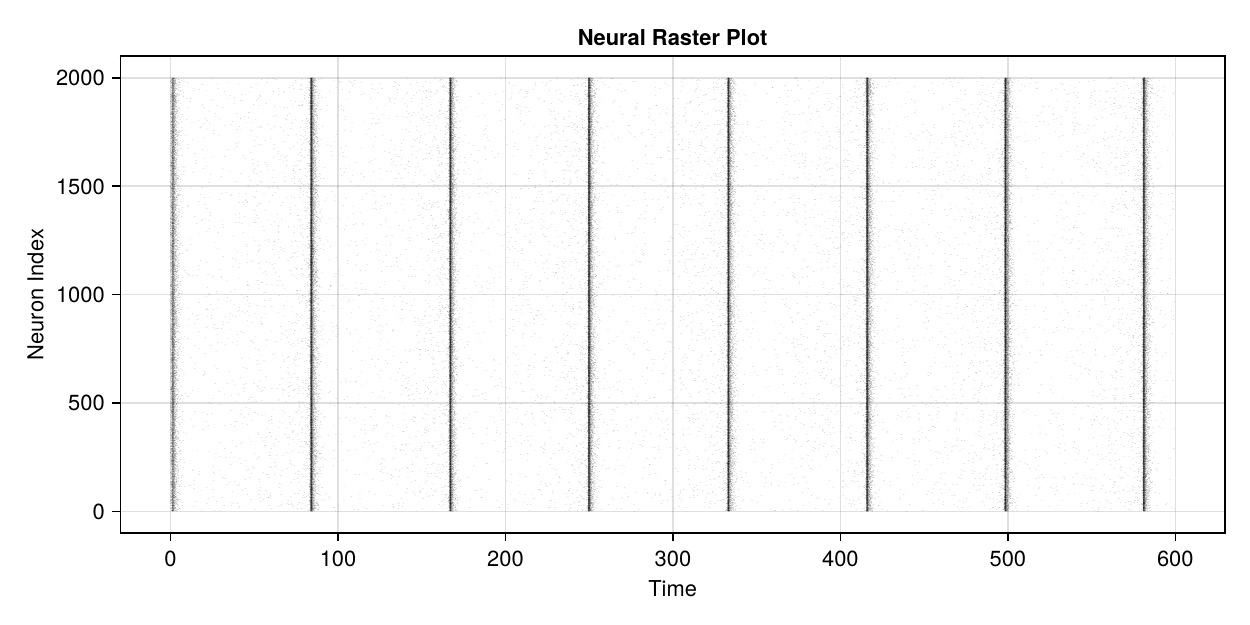}
    \caption{Raster plot of the finite neuron system associated to parameters (\ref{eq:bistable}). Each dot corresponds to a spike of a neuron. We observe a very strong periodic behavior at a frequency that is much lower than the neurons' firing rates.    }
    \label{fig:raster-SOC}
\end{figure}
This example exhibits a bistable behavior -- a feature that would not be observable in the purely one-dimensional case. The slow oscillations observed here are related to a phenomenon called {\it self-organized bistability} known in physics and well described in \cite{PhysRevResearch.2.013318}. We discuss this example in more detail in Section~\ref{sec:num}. 

For the parameters \eqref{eq:bistable}, our analysis shows that the McKean-Vlasov equation~\eqref{eq:VXNLintro} has a unique invariant distribution, corresponding to a value of $\alpha \approx 0.4$. In addition, this invariant distribution is unstable. This is consistent with our main result, Theorem~\ref{theo:2main} (see below), showing that an invariant distribution is (locally) stable provided that the number of solutions of the equation 
\[ \widehat{\Theta}_\alpha(z) = 1, \quad \Re(z) \geq 0, z \in \C, \]
is equal to zero. By the argument principle, this number of solutions is equal to the winding number of the parametric curve $\widehat{\Theta}_\alpha(i \omega), \omega \in \R$, around the point $(1, 0)$ in the complex plane. This winding number is equal to $2$, see Figure~\ref{fig:SOC-nyquist}. Therefore, the equation $\widehat{\Theta}_\alpha(z) = 1$ has two solutions on the half-plane $\{z \in \mathbb{C}: \Re(z) > 0 \}$. This suggests that the invariant distribution is unstable. 

\begin{figure}[!ht]
    \centering
    \includegraphics[width=0.8\textwidth]{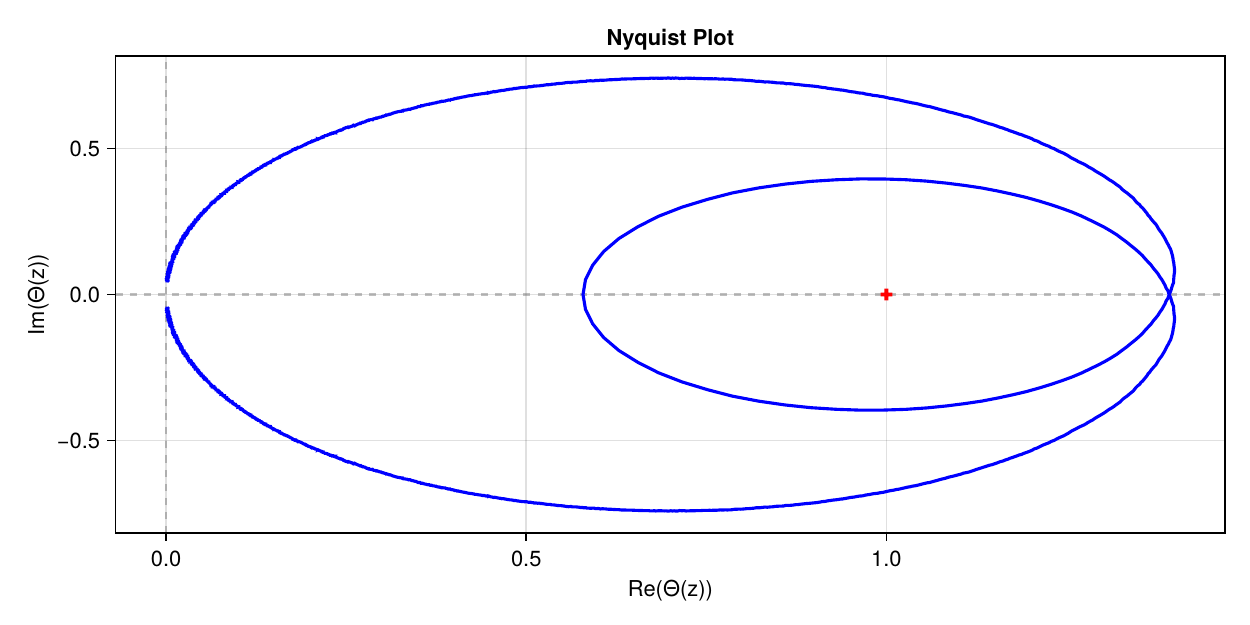}
    \caption{ The Nyquist curve $\widehat{\Theta}_\alpha(i \omega), \omega \in \R$, for the parameters \eqref{eq:bistable} ($\alpha = 0.4$). The winding number of the curve around the point $(1, 0)$ is equal to 2: the equation $\widehat{\Theta}_\alpha(z) = 1$ has therefore two solutions in the half-plane $\Re(z) > 0$. We deduce that the invariant distribution is unstable.}
    \label{fig:SOC-nyquist}
\end{figure}

\subsection{Organization of the paper}
In Section~\ref{sec:2}, we introduce our model and the linearized process and we give the precise model assumptions before stating our main results. We emphasize that Proposition~\ref{prop:invm}, which presents formulas involving the invariant measure of the linearized process, is essential for obtaining our main result, Theorem~\ref{theo:2main}, which gives the precise criterion for the local stability of the invariant measures of (\ref{eq:VXNLintro}) in terms of the Laplace transform of $\Theta_\alpha(t).$ The proofs of our results are gathered in Section~\ref{sec:3}. In Section~\ref{sec:num}, we describe how we obtain the invariant distributions of (\ref{eq:VXNLintro}) numerically and how it is possible to restate our criterion in terms of ordinary differential equations (ODEs). Examples are given in Section~\ref{sec:5}, where we discuss, in particular, the model with parameters (\ref{eq:bistable}) that exhibits a bistable behavior.

\subsection*{Notations}
Throughout this paper, we shall use the following notations. $Lip_1$ denotes the space of all Lipschitz continuous functions $ g : \R \times [0, 1 ] \to \R$ with Lipschitz constant $1$.   We write $\cP ( \R \times [0, 1 ] )$ for the space of probability measures on $ \R \times [0, 1].$ 
For two probability measures $ \nu $ and $\mu $ on $\R \times [0, 1 ], $ we define the bounded-Lipschitz distance
\begin{equation}\label{eq:distanceBL}
d_{BL}(\nu, \mu) = \sup_{g \in Lip_1, \| g\|_\infty \le 1 } \int g(v, x) (\nu-\mu)(\d v, \d x) .
\end{equation}
\section{Statements of the main results}\label{sec:2}
We recall that we are interested in studying the following non-linear equation of McKean-Vlasov type 
\begin{eqnarray}\label{eq:VXNL}
    	\d \bar V^\mu_t & =& %
        b(\bar V^{\mu}_t)\d t + J \E ( \bar X^\mu_t f( \bar V^\mu_t)) \d t - \bar V^\mu_{t-} \int_{ \R_+} \indic{z \leq f(\bar V^\mu_{t-})} N(\d t, \d z) , \nonumber \\
	\d \bar X^\mu_t &=& \frac{1 - \bar X^\mu_t}{\tau} \d t - U \bar X^\mu_{t-} \int_{ \R_+} \indic{z \leq f(\bar V^\mu_{t-})} N(\d t, \d z) ,
\end{eqnarray}
starting from the initial condition $ ( \bar V^\mu_0, \bar X^\mu_0 )= ( V_0, X_0)  \sim \mu$, for $ \mu \in \cP ( \R \times [0, 1 ] )$. In the above equation, $J \in \R $ is the synaptic weight, $U \in (0, 1)$ is a fixed constant, and $ b : \R \to \R , $ and $f: \R \rightarrow \R_+ $ are the drift and the jump rate function, respectively. 

Throughout this paper we shall work with the filtration $({\mathcal F}_t)_{t\geq 0 },$ $ {\mathcal F}_t = \sigma \{ N( A ) : A \in {\mathcal B} ( \R_+ \times \R_+), A \subset [0, t ] \times \R_+  \}\vee \sigma \{ X_0, V_0\}$ where ${\mathcal B}$ denotes the Borel $\sigma$-algebra.

In what follows, we shall write 
\begin{equation}
   \mu_t := \cL (\bar V^\mu_t, \bar X^\mu_t) , \quad  t \geq 0,  
\end{equation}
such that $ \mu_0 = \mu.$ Similarly, we write 
$ \tilde \mu_t := \cL ( \bar V^{\tilde \mu}_t, \bar X^{\tilde \mu}_t) $ for the law of the process starting from the initial law $ \tilde \mu_0 = \tilde \mu.$

\begin{assumption}\label{ass:1}
We assume that $b, f \in C^1 ( \R; \R ) $ with $\| f \|_\infty +  \| f' \|_ \infty  + \| b' \|_\infty < \infty $ and $ b', f'$ are Lipschitz continuous. In particular, $f$ and $b$ are globally Lipschitz continuous, and $f$ is bounded. 
\end{assumption}
We start with the following result on existence of the non-linear equation (\ref{eq:VXNL}).
\begin{proposition}\label{prop:1}
Grant Assumption~\ref{ass:1}. Then the non-linear equation (\ref{eq:VXNL}) possesses a unique strong solution for all $ \mu \in \cP ( \R \times [0, 1]).$ Moreover, for all $ T > 0, $ there exists $C_T > 0 $ such that for all $ 0 \le t \le T,$ for all $ \mu , \tilde \mu \in \cP ( \R \times [0, 1]), $
\begin{equation}
    d_{BL} ( \mu_t, \tilde \mu_t)  \le C_T d_{BL} ( \mu, \tilde \mu ) .  
\end{equation}

\end{proposition}
The proof of this proposition is given in Section~\ref{sec:32} below.

The goal of our article is to study the stability of any invariant state of the non-linear equation (\ref{eq:VXNL}). A main tool for this analysis will be an associated linear equation. 
\subsection*{An associated linear equation}
For any fixed $\alpha \in \R$ and any initial condition $ \mu \in \cP ( \R \times [0, 1]), $ we consider $(V^{\alpha , \mu }_t, X^{\alpha , \mu }_t)$ the solution of the linear  equation
\begin{eqnarray}\label{eq:valphamu}
	\d V^{\alpha , \mu}_t & =& b( V^{\alpha, \mu}_t)\d t + \alpha  \d t - V^{\alpha, \mu}_{t-} \int_{ \R_+} \indic{z \leq f(V^{\alpha, \mu}_{t-})} N(\d t, \d z) , \nonumber \\
	\d X^{\alpha, \mu}_t &=& \frac{1 - X^{\alpha, \mu}_t}{\tau} \d t - U X^{\alpha, \mu}_{t-} \int_{ \R_+} \indic{z \leq f(V^{\alpha, \mu}_{t-})} N(\d t, \d z) , 
\end{eqnarray}
with $ {\mathcal L} ( V^{\alpha , \mu }_0, X^{\alpha , \mu }_0) = \mu.$ 
We also write 
$$ Z_t^{\alpha, \mu} =  \int_{[0, t] \times \R_+} \indic{z \leq f(V^{\alpha, \mu}_{s-})} N(\d s, \d z)$$
for the associated jump process.

Whenever $ \mu = \delta_{ ( v, x) }, $ we will write for short $ ( V_t^{\alpha }, X_t^{\alpha  }),$ and we denote $ \P_{( v, x ) } $ and $ \E_{(v, x) } $ the associated probability measure and expectation under which the process $ ( V^{\alpha }, X^{\alpha  })$ starts from $(v,x)$ at time $0.$ 

Since $b$ is Lipschitz, for any fixed $ \alpha \in \R $ and for any initial condition $v \in \R,$ there exists a unique solution $\varphi^\alpha_t(v)$ of the ODE
\begin{equation}\label{eq:varphialpha}
\frac{\d }{\d t} \varphi^\alpha_t(v) = b( \varphi^\alpha_t(v) ) + \alpha, \quad \varphi^\alpha_0(v) = v. 
\end{equation}
We also write $\psi_t(x)$ for the solution of the ODE
\begin{equation}
	\label{eq:psi_t}
\frac{\d }{\d t} \psi_t(x) =  \frac{1 - \psi_t(v)}{\tau},  \quad  \psi_0(x) = x ,
\end{equation}
given by 
\[ \psi_t(x) = 1 + (x-1)e^{-t/\tau}. \]
We now state a set of assumptions that will be needed in the sequel. 

We start with an assumption that will allow us to couple the first spiking times of two systems, one starting from potential value $v, $ the other from potential value $ \tilde v.$

\begin{assumption}\label{ass:fmin}
We suppose that for any $\alpha \in \R$ and for all $v, \tilde{v} \geq 0,$
\[ \lim_{t \rightarrow \infty} \inf_{v, \tilde{v} \in \R_+} \int_0^t f(\varphi^\alpha_s(v)) \wedge f(\varphi^\alpha_s(\tilde{v})) \d s > 0.\]
\end{assumption}
We will also need the following uniform in time Lipschitz continuity of the flow $ \varphi^\alpha _t .$
\begin{assumption}
For any $ \alpha \in \R, $ there exists a constant $ C = C ( \alpha) > 0 $ such that for all $t \geq 0,  $
\begin{equation}\label{eq:Lipflow} \abs{\varphi^\alpha_t(v) - \varphi^\alpha_t(\tilde{v})} \leq C \abs{v - \tilde{v}}. \end{equation}    
\end{assumption}
Finally, we need a control on the growth rate of the accumulated spiking intensity. 
\begin{assumption}\label{ass:minf}
For any $ \alpha \in \R,$ 
\begin{equation}\label{eq:minf}
\liminf_{t \rightarrow \infty} \inf_{v \in \R_+} \frac{1}{t} \int_0^t f(\varphi^\alpha_s(v)) \d s > 0. \end{equation}
\end{assumption}
\begin{theorem}\label{prop:erglin}
	Grant Assumptions~\ref{ass:1} and~\ref{ass:fmin}-- \ref{ass:minf}. Then the process $ (V^\alpha, X^\alpha) $ is positive Harris recurrent and possesses a unique invariant probability measure $ \mu^\alpha_\infty ( \d  v , \d x). $ Moreover, there exist $C_*, \lambda_* > 0$ such that for all probability measures $ \mu, \tilde \mu $ on $ \R \times [0, 1 ], $
\[ d_{BL} ( \cL  (V_t^{ \alpha, \mu }, X_t^{ \alpha , \mu }),  \cL(V_t^{ \alpha, \tilde \mu }, X_t^{ \alpha , \tilde \mu } ))  \le C_* e^{- \lambda_* t } d_{BL} ( \mu , \tilde \mu ). \]
\end{theorem}

The proof of this result is given in Section~\ref{sec:coupling}.
Let $\sigma_{\alpha} := \lim_{t \rightarrow \infty} \varphi^\alpha_t(0)$. By abuse of notation, in what follows we shall write $ [0, \sigma_\alpha ) $ even when $ \sigma_\alpha < 0.$ In the latter case, $ [0, \sigma_\alpha ) $ denotes the set of all $ v $ such that $ \sigma_\alpha < v \le 0.$ 

Although the invariant distribution is not known explicitly, some of its properties can be described analytically.
\begin{proposition}
	\label{prop:invm}
    The invariant probability measure $ \mu_\infty^\alpha ( \d v , \d x) $ has the following properties.
	\begin{enumerate}
		\item Let $\nu^\alpha_\infty( \d v) := \mu^\alpha_\infty(\d v, [0, 1])$ be the first marginal of $\mu^\alpha_\infty$ with respect to $v.$ This measure possesses a Lebesgue density $ \nu^\alpha_\infty ( v) ,$ and it holds that
			\[ \nu^\alpha_\infty(v) = \frac{\gamma(\alpha)}{b(v)+\alpha} \exp \left(-\int_0^v \frac{f(y)}{b(y) + \alpha} \d y \right) \indica{[0, \sigma_\alpha)}(v), \]
			where the normalizing factor $\gamma(\alpha)$ satisfies $\gamma(\alpha) = \int_{\R \times [0, 1]} f(v) \mu^\alpha_\infty(\d v, \d x)$.
		\item Let $M^\alpha(v) := \lim_{t \rightarrow \infty} \E [ X^\alpha_t~|~ V^\alpha_t = v]$. 
			Then
			\[ M^\alpha(v) = 1 - C^\alpha \exp \left( - \frac{1}{\tau} \int_0^v \frac{\d u}{b(u)+ \alpha} \right), \]
            where the constant $C^\alpha$ is given by 
            \begin{equation}\label{eq:C_alpha}
            C^\alpha = U \left( 1 -  (1-U) \int_0^\infty f(\varphi^\alpha_t(0)) \exp \left(-\int_0^t f(\varphi^\alpha_u(0)) \d u \right) e^{-\frac{t}{\tau}} \d t \right)^{-1} .
            \end{equation}
		\end{enumerate}
\end{proposition}
The proof of this result is given in Section~\ref{sec:proofinvms}. We mention in passing that the formula~\eqref{eq:C_alpha} is reminiscent of similar formulas presented in \cite{Romani2006} and \cite{FonteSchmutz2022}.

\subsection*{Stability of invariant measures of the original non-linear equation}
Let now $\mu_\infty(\d v, \d x)$ be an invariant probability measure of the McKean-Vlasov equation~\eqref{eq:VXNL}.
Let 
\[ \alpha = J \int_{\R \times [0,1]} x f(v) \mu_\infty(\d v, \d x). \]
Then $\mu_\infty = \mu_\infty^\alpha$ is the unique invariant probability measure of the linear process $(V^\alpha_t, X^\alpha_t)$ for this particular choice of $ \alpha. $
Reciprocally, if $\alpha \in \R$ satisfies the fixed-point condition $\alpha = J \int_{\R} f(v) M(v) \nu^\alpha_\infty(v) \d v ,$ then the corresponding $\mu^\alpha_\infty$ is an invariant distribution of the non-linear equation~\eqref{eq:VXNL}. Using this argument, we obtain the following existence result. 
\begin{proposition}
	\label{prop:atleastone}
	The non-linear equation~\eqref{eq:VXNL} has at least one invariant distribution.
\end{proposition}
\begin{proof}
	The assertion follows from the continuity of the function $\alpha \mapsto  J \int_{\R} f(v) M(v) \nu^\alpha_\infty(v) \d v$ and the fact that $f$ and $M$ are bounded.
The continuity of this function is readily seen from the following formulas, which follow from change of variables:
\begin{align*}
&	\frac{1}{\gamma(\alpha)} = \int_0^\infty  \exp\left( -\int_0^t f(\varphi^\alpha_u (0)  ) \d u \right) \d t, \\  
&	\int_{\R} f(v) M(v) \nu^\infty_\alpha(v) \d v = \gamma(\alpha) \left[ 1 - C^\alpha \int_0^\infty K_\alpha(t) e^{-\frac{t}{\tau}} \d t \right]  , 
\end{align*}
where $K_\alpha(t) := f(\varphi^\alpha_t(0)) \exp\left( -\int_0^t f(\varphi^\alpha_u (0)  )  \d u \right)$.
	Therefore, the Brouwer fixed point theorem applies, giving the existence of a fixed-point of this function, and consequently of an invariant distribution.
\end{proof}
In what follows we fix $ b, f, \alpha $ such that Assumptions~\ref{ass:fmin}--\ref{ass:minf} are satisfied and such that $\mu_\infty = \mu_\infty^\alpha.$ 
We want to study the local stability of the invariant measure $ \mu_\infty.$ A strategy for doing so consists of comparing the process $ ( V^{\alpha, \mu_\infty}, X^{ \alpha, \mu_\infty } ),$ the linearized process evolving in stationary regime, with small perturbations thereof, where the fixed input signal $ \alpha $ is replaced by local, time-dependent alternatives of the form $ \alpha + a_t$ with small $a_t $. This comparison will be achieved by evaluating the difference between the Markov generators of the two processes. Comparing these two generators applied to the test function $ g(v,x) = f(v) x $ (the function determining the non-linearity of our equation), 
the important quantity for the study of the local stability turns out to be
\begin{equation}\label{eq:theta}
\Theta_\alpha(t) := J \int_{\R \times [0,1]} \frac{\d }{\d v} \E_{(v,x)}[X^{\alpha}_t f(V^\alpha_t)] \mu_\infty(\d v, \d x) 
\end{equation} 
(see Proposition~\ref{prop:TK} below). 

Let us briefly discuss why $\Theta_\alpha(t)$  is well defined. Applying Theorem~\ref{prop:erglin} to the test function $ g(v, x) = f(v) x, $ we see that for any $v \neq \tilde v, $
\begin{eqnarray*}
    \frac{\left| \E_{(v,x)}[X^{\alpha}_t f(V^\alpha_t)]- \E_{(\tilde v ,x)}[X^{\alpha}_t f(V^\alpha_t)]\right|}{ | v - \tilde v |}  &\le &C_* e^{ - \lambda_* t } \| f \|_{Lip} \vee \|f\|_\infty d_{BL} ( \delta_{ ( v, x)} , \delta_{(\tilde v, x)} ) \frac{1}{ | v - \tilde v |}\\
    &=& C_* e^{ - \lambda_* t } \| f \|_{Lip} \vee \|f\|_\infty, \end{eqnarray*}
since $d_{BL} ( \delta_{ ( v, x)} , \delta_{(\tilde v, x)} ) = | v - \tilde v |.$

We will show in Remark~\ref{rem:315} below that $\frac{\d }{\d v} \E_{(v,x)}[X^{\alpha}_t f(V^\alpha_t)] $ is well defined. So the above argument implies that 
$$ \frac{\d }{\d v} \E_{(v,x)}[X^{\alpha}_t f(V^\alpha_t)] \le C_* e^{ - \lambda_* t } \| f \|_{Lip} \vee \|f\|_\infty , $$
such that for all $ \lambda_\alpha \in (0, \lambda_*), $ 
$$ \sup_{ t \geq 0} e^{ \lambda_\alpha t } | \Theta_\alpha(t) | < \infty . $$
Therefore, the associated Laplace transform 
$$ \hat{\Theta}_\alpha (z) := \int_0^\infty e^{ - zt} \Theta_\alpha (t) \d t $$ 
is well defined on $ D_\alpha = \{ z \in \C, \quad \Re(z) > -\lambda_\alpha \}.$ In Appendix~\ref{app:spectral_heuristic}, we present an intuitive explanation for why the local stability of the non-linear equation is related to the complex roots of the equation $1 - \hat{\Theta}_\alpha (z)=0$. Having all this in mind, we are now able to state our main result. 
\begin{theorem}\label{theo:2main}
Grant Assumptions~\ref{ass:fmin}-- \ref{ass:minf} and assume moreover that there exists $0 < \lambda_\alpha'   \le \lambda_\alpha $ such that
	\[ \forall z \in \C, \quad \Re(z) > -\lambda' _\alpha \implies \hat{\Theta}_\alpha(z) \neq 1. \] 
	Then, $\mu_\infty$ is (locally) stable. More precisely, there exist $ C, \varepsilon  > 0 $ and $ 0 < \lambda < \lambda_\alpha' , $ such that for any  $ \mu_0 \in \cP  ( \R \times [0, 1 ] ) $ with $ d_{BL} ( \mu_0, \mu_\infty ) \le \varepsilon,  $ we have 
    \begin{equation}
        d_{BL} ( \mu_t , \mu_\infty ) \le C e^{ - \lambda t } d_{BL} ( \mu_0, \mu_\infty ) , \end{equation}
        for all $ t \geq 0, $  where $ \mu_t = \cL  (V_t, X_t)  $ with initial law $ \mu_0 = \cL  (V_0, X_0). $
\end{theorem}
\begin{remark}
	It is possible to compute explicitly the value of $\widehat{\Theta}_\alpha(z).$ This remarkable fact will be explained below, see Lemma~\ref{lem:explicit_value_LTheta}. 
\end{remark}

\section{Proofs}\label{sec:3}
\subsection{Notations}
In what follows we shall study linearized versions of the equation (\ref{eq:VXNL}), where we replace the constant $ \alpha $ by some time-dependent input flow $ a_t.$ So given some function 
$a \in C(\R_+; \R)  $ and $ \mu \in \cP ( \R \times [0, 1]), $ we consider $(V^{a, \mu }_{s,t}, X^{a, \mu }_{s, t})_{ 0 \le s \le t < \infty} $ the solution of the linear non-homogeneous equation
\begin{eqnarray}\label{eq:nonhom}
	 V^{a, \mu}_{s,t} & =& \int_s^t b( V^{a, \mu}_{s, r})\d r + \int_s^t a_r  \d r - \int_{ [ s, t ] \times \R_{+} } V^{a, \mu}_{(s,r-)}  \indic{z \leq f(V^{a, \mu}_{s, r-})} N(\d r, \d z) , \nonumber \\
	X^{a, \mu}_{s, t} &=& \int_s^t \frac{1 - X^{a, \mu}_{s, r}}{\tau} \d r - U \int_{ [ s, t ] \times \R_{+} } X^{a, \mu}_{s, r-}    \indic{z \leq f(V^{a, \mu}_{s, r-})} N(\d r, \d z) , 
\end{eqnarray}
for all $ t \geq s, $ with $ {\mathcal L} (V^{a, \mu }_{s, s}, X^{a, \mu }_{s,s}) = \mu.  $

We associate to the above non-homogeneous equation the deterministic flow $\varphi^a_{s,t}(v), 0 \le s \le t < \infty, $ solution of 
\begin{equation}
\frac{\d }{\d t} \varphi^a_{s,t}(v) = b( \varphi^a_{s,t}(v))  + a_t , \quad \varphi^a_{s, s}(v) = v. 
\end{equation}
Moreover, we write 
$$ K_a^\mu ( s, t) = \int_{\R \times [0, 1 ]} f( \varphi_{s, t}^a (v) ) \exp \left( - \int_s^t f( \varphi^a_{s, u } ( v) ) \d u \right) \mu (\d v, \d x ) $$
for the density of the first jump time of $(V^{a, \mu }_{s,t}, X^{a, \mu }_{s, t})_{ t \geq s },$ and $$
H_a^\mu ( s, t) = \int_{\R \times [0, 1 ]}  \exp \left( - \int_s^t f( \varphi^a_{s, u } ( v) ) \d u \right) \mu (\d v, \d x ). $$
Notice that $ H_a^\mu$ and $ K_a^\mu$ depend on $\mu $ only through its first marginal $ \mu ( \d v , [0, 1 ] ). $ Whenever $ \mu = \delta_{(v, x)}, $ we write for short $ H_a^v $ instead of $H_a^{ \delta_{ (v, x) }}$ and $ K_a^v $ instead of $K_a^{ \delta_{ (v, x) }}.$

If the input flow is constant, that is, if there exists some $ \alpha \in \R $ such that $ a_t = \alpha  $ for all $ t \geq 0, $ then we write as before $(V^{\alpha, \mu }_t, X^{\alpha, \mu}_t)$ for the associated time homogenous Markov process. If moreover $ \mu = \delta_{(v,x)}, $ we write for short $ (V^\alpha_t, X^\alpha_t) $ and 
$$ H^v_\alpha (t) = H^{\delta_{(v, x) }}_\alpha (0, t), K^v_\alpha ( t) =K^{\delta_{(v, x ) } }_\alpha (0, t), \varphi_t^\alpha( v) = \varphi_{0, t }^\alpha ( v) ,$$
which do not depend on $ x . $ Recall that in this latter case, $ \P_{(v,x) } $ denotes the probability measure under which the process $(V^\alpha_t, X^\alpha_t) $ starts from $ (V^\alpha_0, X^\alpha_0) = ( v, x) .$ Finally, $ \E_{(v,x)}$ denotes the corresponding expectation and $P^\alpha _t ( (v,x), \cdot ) = {\mathcal L} ( (V_t^\alpha, X_t^\alpha ) | \P_{(v,x) }) $ the associated transition semigroup. 

\subsection{Proof of Proposition~\ref{prop:1}}\label{sec:32}

We first state a lemma without proof which follows along the lines of the proof of Lemma 3.1 in \cite{MNA}. 

\begin{lemma}
Grant Assumption~\ref{ass:1} and fix $ T > 0.$ Then there exists a constant $C_T$ such that for all bounded functions $ g \in Lip_1,$ for all $ a, \tilde a \in C ( \R_+;  \R), $ for all $ 0 \le s \le t \le T, $
\begin{multline*}
| \int_{ \R \times [0, 1 ]} g ( \varphi_{s, t }^a ( v) , \psi_{t-s} (x) ) H_a^{v} (s, t) \mu ( \d v, \d x) - \int_{ \R \times [0, 1 ]} g ( \varphi_{s, t }^{\tilde a} ( v) , \psi_{t-s} ( x)  ) H_{\tilde a}^{v} (s, t) \tilde \mu ( \d v, \d x)    | \\
\le C_T ( 1 + \|g\|_\infty ) \left( \int_s^t | a_u - \tilde a_u | \d u + d_{BL} ( \mu, \tilde \mu ) \right) .
\end{multline*}
\end{lemma}

Conditioning on the first jump time of the process, we obtain furthermore the following result that will be often used in the sequel. 

\begin{lemma}\label{lem:geng}
Let $g : \R \times [0, 1] \to \R$ be measurable and bounded. 
For all $ s \le t, $ we have that 
\begin{multline}\label{eq:geng} 
\E \left[ g  (V_{s, t}^{ a, \mu } , X_{s, t }^{a, \mu } )\right]  = \int_{ \R \times [0, 1] }  g ( \varphi_{s, t }^a (v), \psi_{ t-s } (x) ) H_a^v ( s , t )\mu ( \d v, \d x)   \\
+ \int_{\R \times  [ 0, 1]} \int_s^t K_a^v (s, u)  \E \left[ g   (V_{u , t}^{ a, \delta_{(0, (1- U) \psi_{ u-s} ( x)) } }  , X_{u, t }^{a, \delta_{(0, (1- U) \psi_{u-s} ( x))} } )\right]  \d u  \mu ( \d v,  \d x ) . 
\end{multline}
\end{lemma}

\begin{proof}
Equation~\eqref{eq:geng}  follows from applying the strong Markov property at the first jump time of the process after time $s.$ Details are omitted (see the proof of Lemma 3.2 in \cite{cormier2023stability}). 
\end{proof}

To study the non-linearity appearing in our equation~\eqref{eq:VXNL}, we introduce  
\begin{equation}
    R_{s,t}^a (x)  = \E ( X_{s, t }^{ a , \delta_{(0, x) } } f( V_{s, t }^{a, \delta_{(0, x)}}) ) .  
\end{equation}
Applying~\eqref{eq:geng} with $ \mu = \delta_{(v,x) } $ and $ g (v, x) = x f(v), $ we have the following representation. 
\begin{lemma}\label{lem:33}
\begin{equation}\label{eq:rconvol}
R_{s, t}^a (x)  = K^{0 }_a (s, t) \psi_{ t - s } ( x) + \int_s^t R_{u, t }^a ( (1- U) \psi_{u- s } ( x))   K^{0 }_a ( s, u ) \d u . 
\end{equation}
The latter equation has a unique solution of the form 
\begin{equation}\label{eq:formr}
R_{s, t }^a ( x) = R_{s,t }^{a, 1} + R_{s,t}^{a, 2} x  ,
\end{equation}
where $ R_{s,t }^{a, 1} $ and $ R_{s,t}^{a, 2}$ are given explicitly in~\eqref{eq:R12} below.    \end{lemma}

\begin{proof}
Classical arguments imply that the integral equation~\eqref{eq:rconvol} possesses a unique solution. Using the explicit form \eqref{eq:psi_t}, we see that 
$$ K^{0 }_a (s, t) \psi_{ t - s } ( x) = A_{s, t }^{a,1} + A_{s, t }^{a,2} x ,$$
where $ A_{s, t }^{a, 1} = K^{0 }_a (s, t) ( 1 - e^{ - (t-s) /\tau}) $ and $ A_{s, t }^{a,2} =K^{0 }_a (s, t) e^{ - (t-s)/ \tau }  . $ The same decomposition holds for $ (1 - U) \psi_{s-u} ( x) $ such that 
$$ (1- U) \psi_{u- s } ( x) = B^1_{u- s} + B^2_{u- s} x $$
with 
$$ B^1_{u- s} = ( 1 - U) (1 - e^{ - ( u-s)/\tau} ) ,\;  B^2_{u-s} = ( 1 - U) e^{ - ( u-s)/\tau} .$$
We look therefore for solutions of the form 
$$ R_{s,t }^a (x) = R_{s, t}^{a, 1} + R_{s, t }^{a, 2} x ,$$
where $ R^{a, 1} $ and $ R^{a, 2 }$ solve 
\begin{eqnarray}\label{eq:R12}
R^{a, 1 }_{s, t }& =& A_{s, t}^{a,1} + \int_s^t K^v_a ( s, u ) R^{a, 1}_{u, t} \d u + \int_s^t K^v_a (s, u ) R^{a, 2 }_{u, t} B^1_{u-s} \d u \nonumber \\
R^{a, 2 }_{s, t }&=& A^{a,2}_{s, t } + \int_s^t K^v_a (s, u ) R^{a, 2 }_{u, t } B^2 _{ u-s} \d u . 
\end{eqnarray}
The second equation is a standard non-homogenous Volterra integral equation and can be solved via a standard Neumann series. Once $R^{a, 2} $ is known, the first equation in $ R^{a, 1} $ is also a standard Volterra equation and has a unique solution. 
\end{proof}

We now study \begin{equation}
    R_{s,t}^a (v, x)  = \E ( X_{s, t }^{ a , \delta_{(v, x) } } f( V_{s, t }^{a, \delta_{(v, x)}}) ) .  
\end{equation}

\begin{lemma}
We have that 
\begin{equation}\label{eq:Ravx}
R_{s,t}^a (v, x) = K^v_a (s, t) \psi_{ t-s}(x) + \int_s^t K^v_a (s, u) R_{u, t }^a ( (1- U) \psi_{ u-s } (x) ) \d u . 
\end{equation}
As a consequence, 
\begin{equation}\label{eq:expl} R_{s, t }^a (v, x) = S^{a, 1 }_{s, t} (v) + S_{s, t}^{a,2} (v) x , 
\end{equation}
where 
$$ S^{a, 1 }_{s, t} (v) = K_a^v ( s, t) \psi_{t-s} (0) + \int_s^t K^v_a (s, u ) R^{a, 1}_{u,t} \d u + (1 - U) \int_s^t K_a^v (s, u ) \psi_{u-s} ( 0)  R^{a, 2 }_{u, t } \d u $$
and 
$$ S^{a, 2 }_{s, t } ( v) = K^v_a (s, t) e^{ (t-s)/\tau} + (1- U) \int_s^t K^v_a (s, u ) e^{ - (u-s)/\tau} R^{a, 2 }_{u, t} \d u . $$

\end{lemma}

The following auxiliary results will be useful in the sequel. 
\begin{lemma}
There exists a constant $C_T$ such that for all $ a, \tilde a \in C( \R_+ ; \R ) , $ for all $ 0 \le s \le t \le T, $ and $ v \in \R ,$
$$ | K^v_a (s,t ) - K^v_{\tilde a } (s, t) | \le C_T \int_s^t | a_u - \tilde a_u| \d u .$$ 
In particular, 
$$ | A^{ a, k}_{s, t } - A^{\tilde a , k }_{s, t } | \le C_T \int_s^t | a_u - \tilde a_u| \d u , $$
for $ k = 1, 2. $
\end{lemma}
The proof of this lemma follows directly from the definition of $ K^v , A^{a, k}.$

\begin{lemma}
There exists a constant $C_T$ such that for all $ a, \tilde a \in C( \R_+; \R ) , $ and for all $ 0 \le s \le t \le T, $ for $ k = 1, 2, $
$$ | R^{a, k }_{s, t } - R^{ \tilde a , k}_{s, t } | \le C_T \int_s^t | a_u - \tilde a_u | \d u .$$
\end{lemma}

\begin{proof}
Let us first consider the case $ k=2.$ We have that 
\begin{multline*}
R^{a, 2 }_{s, t } - R^{ \tilde a , 2}_{s, t } = A^{a, 2 }_{s, t } - A^{ \tilde a , 2}_{s, t } \\
 + \int_s^t \left[ K^v_a (s, u ) - K^v_{\tilde a} (s, u )   \right] R^{ a, 2}_{u, t } B_{u-s}^2 \d u 
 + \int_s^t K^v_{\tilde a} (s, u )\left[ R^{a, 2 }_{u, t } - R^{ \tilde a , 2}_{u, t } \right] B^2_{u-s} \d u , 
\end{multline*}
such that the assertion follows from Gronwall's lemma. Once the assertion is proven for $ R^{a, 2}, $ using~\eqref{eq:R12} and Gronwall's lemma once more, it also follows for $ R^{a, 1}. $
\end{proof}

Using~\eqref{eq:Ravx}, we deduce similarly the following result. 
\begin{lemma}\label{lem:36}
There exists a constant $C_T$ such that for all $ a, \tilde a \in C( \R_+ ; \R ) , $ and for all $ 0 \le s \le t \le T, $
$$ |R_{s,t}^a (v, x)- R_{s,t}^{\tilde a}  (v, x)| \le C_T \int_s^t | a_u - \tilde a_u | \d u . $$ 
\end{lemma}

Finally, we notice that for any fixed $a \in C( \R_+; \R) , $ due to the explicit structure~\eqref{eq:expl}, 
$$ \R \times [0, 1 ] \ni (v, x) \mapsto R_{s,t}^a (v, x) $$
is bounded and Lipschitz continuous, with 
$$ \sup_{ v, x} | R_{s, t }^a (v, x) | + \sup_{ 0 \le s \le t \le T } \| R_{s, t }^a \|_{ Lip } \le C_T $$
for a constant $C_T$ that does not depend on $a. $

Let now $ R_a^{\mu  } ( s, t ) = \int_{ \R \times [0, 1 ] } R^a_{s, t} (v, x)  \mu ( \d v, \d x). $ As a consequence of the above arguments we  deduce the following result. 
\begin{lemma}\label{lem:38}
There exists a constant $C_T$ such that for all $ a, \tilde a \in C( \R_+ ; \R ) , $ and for all $ 0 \le s \le t \le T, $
$$  | R_a^{\mu  } ( s, t )- R_{\tilde a}^{ \tilde \mu  } ( s, t )| \le C_T \int_s^t | a_r - \tilde a_r|\d r + C_T d_{BL} ( \mu, \tilde \mu ) . $$
\end{lemma}

\begin{proof}
We first show the inequality for $ \mu = \tilde \mu = \delta_{ (v, x) } $ for some $ x \in [0, 1 ], v \in \R. $ Then 
$$R_a^{\mu  } ( s, t )- R_{\tilde a}^{ \tilde \mu  } ( s, t ) = R_{s, t }^a (v,  x) - R_{s,t }^{\tilde a } (v, x) , $$
and the result follows from Lemma~\ref{lem:36}. Integrating the inequality with respect to $ \mu $ yields the result for arbitrary $ \mu ,$ with $ \tilde \mu = \mu .$ Finally, we have that 
$$ R_{\tilde a}^{  \mu  } ( s, t )- R_{\tilde a}^{ \tilde \mu  } ( s, t ) = \int_{ \R \times [0, 1]} R^a_{s, t} (v, x) ( \mu  -  \tilde \mu)  ( \d v, \d x) ,  $$
and the result follows from the definition of the distance $ d_{BL} $ and the fact that $ R^a_{s, t} (v, x)$ is bounded and Lipschitz continuous. 

\end{proof}

In what follows, we fix a function $g \in Lip_1 $ such that $ \|g \|_\infty \le 1, $ and we extend the above arguments to the study of  $\E \left[ g  (V_{s, t}^{ a, \mu } , X_{s, t }^{a, \mu } )\right]  .$ 

\begin{lemma}
For any $g \in Lip_1 $ such that $ \|g \|_\infty \le 1, $ the function $ x \mapsto \E \left[ g  (V_{s, t}^{ a, \delta_{(0, x)} } , X_{s, t }^{a, \delta_{(0, x)} } )\right] =: R_{s, t}^{a, g } ( x) $ is Lipschitz continuous with Lipschitz constant 
$$ \sup_{ x \neq x' }\frac{| R_{s, t}^{a, g } ( x) - R_{s, t}^{a, g } ( x')| }{|x-x'|} \le e^{ \|f\|_\infty (t-s) }. $$
If moreover $g \in C^1 ( \R \times [0, 1 ]; \R),$ then $x \mapsto R_{s, t}^{a, g } ( x)$ is also differentiable. 
\end{lemma}

\begin{proof}
Applying Lemma~\ref{lem:geng} to $ \mu = \delta_{(0, x ) }, $ we see that $ x \mapsto \E \left[ g  (V_{s, t}^{ a, \delta_{(0, x)} } , X_{s, t }^{a, \delta_{(0, x)} } )\right] =: R_{s, t}^{a, g } ( x) $ satisfies an integral equation 
of the same type as (\ref{eq:rconvol}) and that 
\begin{multline*}
    | R_{s, t}^{a, g } ( x) - R_{s, t}^{a, g } ( x')| \le H_a^0 (s, t) | x - x' | \\
    + \int_s^t K_a^0 ( s, u ) |  R_{u, t}^{a, g } ( (1 - U) \psi_{ u-s} (x) ) - R_{u, t}^{a, g } ( (1 - U) \psi_{ u-s} (x'))|  \d u .\end{multline*} 
The result then follows from iterating this inequality, by upper bounding $H_a^0 (s, t) \le 1, $ $ K_a^0 (s, t) \le \|f\|_\infty$ and $|(1 - U) \psi_{ u-s} (x) ) -  ( (1 - U) \psi_{ u-s} (x')|\le |x-x'|. $   

Finally, the differentiability of $ R_{s, t}^{a, g } ( x) $ follows analogously, using that 
\[ x \mapsto g( \varphi_{s,t }^a ( 0), \psi_{t-s} ( x) ) \in C^1 ( [0, 1]; \R ). \]
\end{proof}

Using once more Lemma~\ref{lem:geng}, we may then deduce that for all $ 0 \le s \le t \le T, $
$$ (v, x) \mapsto \E \left[ g  (V_{s, t}^{ a, \delta_{(v, x)} } , X_{s, t }^{a, \delta_{(v, x)} } )\right] =: R_{s, t}^{a, g } (v, x)$$
is Lipschitz continuous, with Lipschitz constant $C_T $ depending only on $ T,$ and that this latter function is moreover differentiable, if $g \in C^1 ( \R \times [0, 1 ] ; \R). $ 

Finally, similar arguments as those of Lemma~\ref{lem:36} imply the following result.
\begin{lemma}
There exists a constant $C_T$ such that for all $ a, \tilde a \in C( \R_+; \R ) , $ and for all $ 0 \le s \le t \le T, $ for all functions $g \in Lip_1 $ such that $ \|g \|_\infty \le 1, $
$$ |R_{s,t}^{a,g} (v, x)- R_{s,t}^{\tilde a, g}  (v, x)| \le C_T \int_s^t | a_u - \tilde a_u | \d u . $$ 
\end{lemma}

Analogously to the proof of Lemma~\ref{lem:38} we deduce from this the following result. 

\begin{lemma}
There exists a constant $C_T$ such that for all $ a, \tilde a \in C( \R_+; \R ) , $ and for all $ 0 \le s \le t \le T, $ for all functions $g \in Lip_1 $ such that $ \|g \|_\infty \le 1, $
$$ | \E \left[ g  (V_{s, t}^{ a, \mu } , X_{s, t }^{a, \mu } )\right] - \E \left[ g  (V_{s, t}^{ \tilde a, \tilde \mu } , X_{s, t }^{\tilde a, \tilde \mu } )\right]| \le C_T \left( \int_s^t | a_u - \tilde a_u | \d u + d_{BL} ( \mu , \tilde \mu ) \right) . $$
\end{lemma}

We are now able to give the proof of Proposition~\ref{prop:1}.
\begin{proof}[Proof of Proposition~\ref{prop:1}]
The existence of a solution is not difficult and follows analogously to~\cite{CTV}. Let $a_t = \E f( V_t^\mu ) X_t^\mu  .$ Applying Ito's formula, it follows that $ a \in C ( [ 0, T]; \R). $ The proof then follows from the observation that $ (V^\mu, X^\mu) $ is solution of~\eqref{eq:nonhom} with $a.$ See the proof of Theorem 2.2 in \cite{MNA} for details.    
\end{proof}

\subsection{Proof of Theorem~\ref{prop:erglin}}\label{sec:coupling}
Recall that the process $(V^{\alpha , \mu }_t, X^{\alpha , \mu }_t)$ is defined in~\eqref{eq:valphamu} and that we write $ (V^\alpha, X^\alpha), $ whenever the process starts from some fixed initial conditions $ (v, x), $ that is, when $ \mu = \delta_{(v,x)}. $ Recall that $P^\alpha _t ( (v,x), \cdot ) = {\mathcal L} ( (V_t^\alpha, X_t^\alpha ) | \P_{(v,x) }) .$ We start with the following preliminary result. 

\begin{proposition}\label{theo:invm}
Grant Assumptions~\ref{ass:fmin}-- \ref{ass:minf}. Then the process $ (V^\alpha, X^\alpha ) $ is positive Harris recurrent and possesses a unique invariant probability measure $ \mu_\infty (\d v, \d x) .$ Moreover, there exist $ \lambda , C  > 0 $ such that 
for all $ v, \tilde v \in \R, x, \tilde x \in [0, 1 ], $
\begin{eqnarray} \label{eq:exprate}
d_{BL} ( P^\alpha _t ( (v,x), \cdot ), P^\alpha _t ( (\tilde v,\tilde x), \cdot )) 
&\le &C_* e^{ - \lambda_* t }, \nonumber \\
d_{BL} ( P^\alpha _t ( (v,x), \cdot ),  \mu^\alpha_\infty )  &\le &  C e^{ - \lambda t } . 
\end{eqnarray}
Moreover, for all $ x , \tilde x \in [0, 1] ,$ 
\begin{equation}\label{eq:exprate2}
d_{BL} ( P^\alpha _t ( (0,x), \cdot ), P^\alpha _t ( (0,\tilde x), \cdot ) ) \le  e^{ - \frac{1}{\tau} t } |x - \tilde x |.    
\end{equation}
\end{proposition}

To prove the above result, we consider two solutions $ (V^\alpha, X^\alpha)$ starting from $ (v, x)$ and $ ( \tilde V_t^\alpha, \tilde X_t^\alpha) $ starting from $ ( \tilde v, \tilde x ), $ where $ x , \tilde x \in [0, 1 ]. $ %
We take the synchronous coupling of the two processes, that is, we use the same underlying Poisson random measure and make the two processes jump together as often as possible.
We introduce the coupling time 
$$ \tau_c = \inf \{ t \geq 0 :  \Delta Z_t =  \Delta \tilde Z_{t}  \}, $$
with 
$$ Z_t = \int_{ [0, t ] \times \R_+} \indic{ z \le f ( V_{s-} ) } N (\d s, \d z) , $$
and 
$$ \tilde Z_t = \int_{ [0, t ] \times \R_+} \indic{ z \le f ( \tilde V_{s-} ) } N (\d s, \d z) .$$
$ \tau_c $ is the first synchronous jump of the processes $ Z $ and $\tilde Z. $ At time $ \tau_c,$ both processes $ V$ and $ \tilde V$ are synchronously reset to $0  $ and then stay together for all future times, since their dynamics is driven by the same underlying Poisson random measure and since they have the same drift. In other words, starting from time $ \tau_c, $ all jumps of $ Z $ and $ \tilde Z$ will be synchronous jumps.

In what follows, we write $ \P_{(v,\tilde v, x, \tilde x)}$ for the probability measure corresponding to the above coupling. $ \E_{(v,\tilde v, x, \tilde x)}$ denotes the corresponding expectation. Recall the definition of the flow $ \varphi^\alpha $ in~\eqref{eq:varphialpha}.

\begin{proposition}
There exists $\kappa > 0$ such that  
\[ \sup_{ x, \tilde x \in [0, 1 ] } \sup_{ v, \tilde v \in \R_+} \E_{(v,\tilde v, x, \tilde x)} e^{ \kappa \tau_c} < \infty.\]
\end{proposition}

\begin{proof}
Recall that by Assumption~\ref{ass:fmin}, 
\[ \lim_{t \rightarrow \infty} \inf_{v, \tilde{v} \in \R_+} \int_0^t f(\varphi^\alpha_s(v)) \wedge f(\varphi^\alpha_s(\tilde{v})) \d s > 0,\]
and that $f$ is bounded and Lipschitz. So there exist $T> 0 $ and $c> 0 $ such that for all $v, \tilde{v},$
\[ \int_0^T f(\varphi^\alpha_s(v)) \wedge f(\varphi^\alpha_s(\tilde{v})) \d s  \geq c.  \]
Between each $kT$ and $(k+1)T,$ $ k \geq 0, $ we make a coupling attempt between the two processes.
It succeeds with probability at least
\[ \P_{(v,\tilde v, x, \tilde x)} ( \tau_c \in ] kT, (k+1) T ]| {\mathcal F}_{kT} ) \geq \int_0^T f(\varphi^\alpha_s(V^{\alpha}_{kT})) \wedge f(\varphi^\alpha_s(\tilde V^{\alpha}_{kT})) \exp(-\| f \|_\infty s) \d s. \]
This probability is lower bounded by
\[ e^{-\| f\|_\infty T} c > 0 . \]
The conditional Borel-Cantelli lemma then allows us to conclude. 
\end{proof}

\begin{proof}[Proof of Proposition~\ref{theo:invm}]
At time $\tau_c,$ both processes $ V^\alpha $ and $ \tilde V^\alpha $ are synchronously reset to $0 $ and then stay together for all times. 

Notice that $ X^\alpha_t, \tilde X^\alpha_t \in [0, 1] $ by construction. Therefore, 
\begin{equation}\label{eq:distx}
     | X^\alpha_t - \tilde X_t^\alpha | \le 1_{\{ \tau_{c} > t/2\}} + e^{ - \frac{1}{\tau} ( t - \tau_{c} ) } | X^\alpha_{\tau_{c}} - \tilde X_{\tau_{c}}^\alpha | 1_{\{ \tau_{c} \le  t/2\}} \le 1_{\{ \tau_{c} > t/2\}} + e^{ - \frac{1}{\tau} (t/2)   } ,
\end{equation}     
where we have used that $ | X_s^\alpha - \tilde X_s^\alpha | \le 1$ for all $s.$ Therefore, 
$$ \E | X^\alpha_t - \tilde X_t^\alpha | \le C e^{ - \kappa (t/ 2)} + e^{ - \frac{1}{\tau} (t/2)} . $$
Let now $g \in Lip_1 $ such that $g$ is bounded by $1.$ We use that $ | g(v, x ) - g( \tilde v , \tilde x) | \le \indic{v \neq \tilde v  } + | g ( v, x) - g(v, \tilde x) |$ to obtain that 
$$ | g ( V_t^\alpha, X_t^\alpha ) - g ( \tilde V_t^\alpha, \tilde X_t^\alpha ) | \le \indic{V_t^\alpha \neq V_t^\alpha } + | X_t^\alpha - \tilde X_t^\alpha | \le \indic{  \tau_c > t } + | X_t^\alpha - \tilde X_t^\alpha | , $$
such that
$$ \E_{(v,\tilde v, x, \tilde x)} | g ( V_t^\alpha, X_t^\alpha ) - g ( \tilde V_t^\alpha, \tilde X_t^\alpha ) | \le \P (  \tau_c > t) + C e^{ - \kappa (t/ 2)} + e^{ - \frac{1}{\tau} (t/2)} . $$
This implies that, for convenient positive constants $ C_*, \lambda_*, $ 
$$ | \E_{ (v, x) } g( V_t^\alpha, X_t^\alpha)  -\E_{ (\tilde v, \tilde v) } g( V_t^\alpha, X_t^\alpha) | \le C_* e^{ - \lambda_* t }.$$

Classical arguments imply that $ ( V^\alpha, X^\alpha)$ is Harris recurrent, possessing a unique invariant probability distribution $ \mu_\infty (\d v, \d x). $ Integrating the above inequality against 
$ \mu_\infty (d \tilde v, d\tilde x)$ concludes the proof of (\ref{eq:exprate}). 

To prove (\ref{eq:exprate2}), notice that when $ v = \tilde v = 0, $ both processes $ V^\alpha_t$ and $\tilde V_t^\alpha$ are immediately coupled such that $ \tau_c = 0. $ The assertion then follows from (\ref{eq:distx}).
\end{proof}

We are now able to give the proof of Theorem~\ref{prop:erglin}. 

\begin{proof}[Proof of Theorem~\ref{prop:erglin}]
We fix some $v, x, \tilde{v}, \tilde{x}$ and we show that 
\begin{equation}\label{eq:step1} 
d_{BL} (P_t^\alpha ( (v,x) , \cdot ) , P^\alpha_t ( (\tilde v, \tilde x), \cdot ))  \leq C_* e^{-\lambda_* t} (|v-\tilde{v}|\wedge 1  + |x - \tilde{x}|). \end{equation}
To do so, we take some test function $g \in Lip_1$ with $ \|g\|_\infty \le 1 $ and we define
\[ R_t(v, x) := \E_{(v,x)} g(V^{\alpha}_t, X^{\alpha}_t). \]
Recall that $K^v(t)$ denotes the probability density of the first jump time of $Z^\alpha_t$, under $ \P_{(v,x)}, $ and that this density is given by
\begin{equation}\label{eq:Kt}
    	K^v(t) := - \frac{\d }{\d t} \P(Z^\alpha_t = 0) = f(\varphi^\alpha_t(v)) \exp\left( -\int_0^t{f(\varphi^\alpha_s(v)) \d s} \right).   
\end{equation}
The same arguments as those used in Section~\ref{sec:32} imply that 
\begin{align*} R_t(v, x) - R_t(\tilde{v}, \tilde{x}) &= H^v(t) g(\varphi^\alpha_t(v), \psi_t(x)) - H^{\tilde{v}}(t) g(\varphi^\alpha_t(\tilde{v}), \psi_t(\tilde{x})) \\ &\quad  + \int_0^t (K^v(s) - K^{\tilde{v}}(s))  R_{t-s}(0, (1-U)\psi_s(x)) \d s \\
&\quad  + \int_0^t K^{\tilde{v}}(s) (R_{t-s}(0, (1-U)\psi_s(x)) - R_{t-s}(0, (1-U)\psi_s(\tilde{x}))) \d s,
\end{align*}
where $ H^v ( t) = \exp \left(- \int_0^t f ( \varphi^\alpha_s ( v) ) \d s  \right). $ 
Since
\[ \abs{g(v, x) - g(\tilde{v}, \tilde{x})} \leq |v-\tilde{v}|\wedge 1  + |x - \tilde{x}|, \]
we have, recalling (\ref{eq:Lipflow}), 
\[ | g(\varphi^\alpha_t(v), \psi_t(x)) -  g(\varphi^\alpha_t(\tilde{v}), \psi_t(\tilde{x})) | \le \left( C |v-\tilde{v}|\right) \wedge 1  + | x- \tilde x| ,  \]
since $ | \psi_t(x) -  \psi_t(\tilde{x})|\le | x - \tilde x|. $

In what follows, we denote by $ F$ a constant such that $ \|f\|_{Lip} , \|f\|_\infty \le F. $ In particular, we have that $ | f (x) - f(y ) | \le F ( |x-y | \wedge 1 ). $ By (\ref{eq:minf}) there exist $t_* > 0 $ and $ f_* > 0 $ such that for all $ t \geq t_*  $ and for all $v,$ 
$$ \int_0^t f( \varphi^\alpha_s ( v)) \d s \geq f_* t .$$ 
Lower bounding $ f_* t $ by $ f_* ( t- t_*) $ and choosing $c_* = e^{ f_* t_*}, $ we deduce that for all $v,$ and for all $ t \geq 0, $ 
$$ H^v ( t) \le c_* e^{ - f_* t } \mbox{ and } | H^v ( t) - H^{ \tilde v } (t) | \le c_* e^{ - f_* t }  t F ( | v - \tilde v| \wedge 1 ) . $$
Therefore, 
\begin{multline*}
    |H^v(t) g(\varphi^\alpha_t(v), \psi_t(x)) - H^{\tilde{v}}(t) g(\varphi^\alpha_t(\tilde{v}), \psi_t(\tilde{x})) | \le \\
  \le   H^v (t) | g(\varphi^\alpha_t(v), \psi_t(x)) -  g(\varphi^\alpha_t(\tilde{v}), \psi_t(\tilde{x}))|+ | g(\varphi^\alpha_t(\tilde{v}), \psi_t(\tilde{x}))| | H^v ( t) - H^{ \tilde v } (t )| \\
     \leq c_* e^{ - f_* t } \left( ( C + F t ) \left( | v - \tilde v|\wedge 1 \right)  +| x - \tilde{x}| \right)\\
     \le C e^{ - (f_* / 2) t } ( | v - \tilde v|\wedge 1   +| x - \tilde{x} | ),
\end{multline*}
where we took some convenient constant to obtain the last line. 
In addition, using~\eqref{eq:exprate}, we have that
\[ R_{t}(0, x)  = \E_{(0, x)} g(V^\alpha_t, X^\alpha_t) = \langle g, \mu_\infty \rangle + \xi(t, x),  \]
where the function $\xi(t, x)$ satisfies $\sup_{x} \sup_{t \geq 0} e^{\lambda t} \abs{\xi(t, x)} < \infty$. So 
\begin{multline*}
\int_0^t  (K^v(s) - K^{\tilde{v}}(s))  R_{t-s}(0, (1-U)\psi_s(x)) \d s =  \langle g, \mu_\infty \rangle  \int_0^t  (K^v(s) - K^{\tilde{v}}(s)) \d s  +\\
+ \int_0^t  (K^v(s) - K^{\tilde{v}}(s) ) \xi (t-s, (1-U)\psi_s(x) ) \d s =\\
= \langle g, \mu_\infty \rangle (H^v ( t) - H^{ \tilde v}(t)) +\int_0^t  (K^v(s) - K^{\tilde{v}}(s) ) \xi (t-s, (1-U)\psi_s(x) ) \d s  =: T_1 + T_2 .
\end{multline*}
Here we have used that $\int_0^t  K^v(s) \d s = 1 - H^v (t) .$ We have that $ | T_1 | \le C e^{ - (f_*/2) t } | v - \tilde v |\wedge 1  .$
Moreover, 
\begin{multline*}
|K^v(s) - K^{\tilde{v}}(s) |\le | f( \varphi^\alpha_s( v) ) ( H^v (s) - H^{\tilde{v}} (s) )| + |H^{\tilde{v}} (s) ( f( \varphi^\alpha_s ( v) )- f( \varphi^\alpha_s ( \tilde{ v}) |
\le \\
\| f \|_\infty c_* e^{ - f_* s } F s \left( | v - \tilde v|\wedge 1 \right)  + c_* C e^{ - f_* s } F \left( |v - \tilde v|\wedge 1 \right)  \le C e^{- (f_*/2) s } \left(  | v - \tilde v| \wedge 1\right) , 
\end{multline*}
where we choose yet another constant such that the last estimate holds true for all $s.$ 

As a consequence, assuming without loss of generality that $ f_* / 2 \neq \lambda, $
$$|T_2| \le C \left( | v - \tilde v|\wedge 1\right)  \int_0^t e^{- (f_*/2) s } e^{ - \lambda ( t- s) } \d s  \le \frac{C}{ | f_*/ 2 - \lambda |} e^{ - (( f_*/2 )\wedge \lambda) t }  | v - \tilde v |\wedge 1 . $$
Finally, using (\ref{eq:exprate2}), we have that 
$$ | R_{t-s}(0, (1-U)\psi_s(x)) - R_{t-s}(0, (1-U)\psi_s(\tilde{x})) | \le e^{ - \frac{1}{\tau} (t-s) } | x - \tilde x |, $$
such that we obtain similarly 
$$ |\int_0^t K^{\tilde{v}}(s) (R_{t-s}(0, (1-U)\psi_s(x)) - R_{t-s}(0, (1-U)\psi_s(\tilde{x}))) \d s| \le C \int_0^t e^{ - f_* s} e^{ - \frac{1}{\tau} ( t-s) }  
| x - \tilde{x}| \d s , $$
and the conclusion follows for some convenient constants $ C_*, \lambda_* > 0.$ 

Integrating the upper bound $ C e^{ - \lambda t } ( |v - \tilde v | \wedge 1) + | x - \tilde x|$ against $ ( \mu - \tilde \mu ) ( \d v, \d x ) $ then yields the assertion for general initial conditions.

\end{proof}

\subsection{Proof of Proposition~\ref{prop:invm}}\label{sec:proofinvms}
We are now able to give the 
\begin{proof}[Proof of Prop.~\ref{prop:invm}]
The proof of item 1. is standard and therefore omitted. 

To prove item 2., take the test function $g(v, x) = e^{ - a v} x,$ for some $ a > 0.$  Then, denoting $A^\alpha $ the generator of the linearized process with input $\alpha,$ we have 
$$ A^\alpha g (v,x) = - a g(v,x)   ( b(v) + \alpha ) + e^{ - a v} \frac{1-x}{\tau}  + f (v) [ (1- U) x - g(v, x) ] ,  $$
where $b(v) = \bar V - v.$

Integrating against the invariant measure $\mu_\infty (\d v, \d x) = \nu_\infty (v) \mu_\infty ( v, \d x) \d v $ %
and writing $ M( v) = \int_{[0, 1]} x \mu_\infty (v, \d x) $ gives 
\begin{multline} a \int_0^\infty \nu_\infty ( v) e^{ - a v } ( b(v) + \alpha ) M(v)\d v   = \int_0^\infty  \nu_\infty ( v) e^{ -a v}  \frac{ 1 - M(v) }{\tau} \d v - U \int_0^\infty  f(v) M(v) \nu^\infty (v) \d v \\
 + \int_0^\infty  f(v) [ 1 - e^{ - a v} ] M(v) \nu_\infty (v) \d v. 
\end{multline}
Here, we have separated the term $ f (v) [ (1- U) x - g(v, x) ]  = - U f(v) x + f (v) x[ 1 - e^{ - a v}  ].$ 
 
Observe that $ 1 - e^{ - av} = a \int_0^v e^{ - a u } \d u ,$ such that we can rewrite the last expression as 
$$\int_0^\infty  f(v) [ 1 - e^{ - a v} ] M(v) \nu^\infty (v) \d v =  a \int_0^\infty e^{ - a u} \left( \int_u^\infty f(v) M(v) \nu_\infty (v) \d v \right) \d u .$$   
Therefore, 
\begin{multline} a \int_0^\infty  \nu_\infty ( v) e^{ - a v } ( b(v) + \alpha ) M(v)  = a \int_0^\infty e^{ - a u} \left( \int_u^\infty f(v) M(v) \nu_\infty (v) \d v \right) \d u \\
+ \int_0^\infty \nu_\infty ( v) e^{ -a v}  \frac{ 1 - M(v) }{\tau} \d v - U \int_0^\infty  f(v) M(v) \nu_\infty (v) \d v .
\end{multline}
We now let $ a \to 0 $ and deduce that 
\begin{equation}\label{eq:ufm}  U \int_0^\infty  f(v) M(v) \nu_\infty (v) \d v = \int_0^\infty  \nu_\infty ( v)  \frac{ 1 - M(v) }{\tau} \d v .
\end{equation}
Therefore, 
$$ \int_0^\infty  \nu_\infty ( v) e^{ -a v}  \frac{ 1 - M(v) }{\tau} \d v - U \int_0^\infty  f(v) M(v) \nu_\infty (v) \d v = \int_0^\infty  \nu_\infty ( v)[  e^{ -a v} - 1 ] \frac{ 1 - M(v) }{\tau} \d v , $$
which can once more be rewritten as 
$$ - a \int_0^\infty e^{ - a u } \left( \int_u^\infty \nu_\infty ( v) \frac{ 1 - M(v) }{\tau} \d v \right) \d u .$$ 
Dividing everything by $a, $ we therefore conclude that we have equality of Laplace transforms 
\begin{multline} 
 \int_0^\infty \nu_\infty ( v) e^{ - a v } ( b(v) + \alpha ) M(v) \d v  =  \int_0^\infty e^{ - a u} \left( \int_u^\infty f(v) M(v) \nu_\infty (v) \d v \right) \d u \\
 - \int_0^\infty e^{ - a u } \left( \int_u^\infty \nu_\infty ( v)\frac{ 1 - M(v) }{\tau} \d v \right) \d u. 
\end{multline}
On the left hand side, we have the Laplace transform of the signed measure $ \nu_\infty (v) [ b(v)  + \alpha ] M(v)\d v  .$ Here the only unknown quantity is $M(v).$ On the right hand side, we have the Laplace transform of a measure having the Lebesgue density 
$$  u \mapsto \int_u^\infty [ f(v) M(v) -  \frac{ 1 - M(v) }{\tau}  ]  \nu_\infty (v) \d v  .$$ 
So we may deduce from this that 
$$ \nu_\infty ( u) ( b(u) + \alpha ) M(u) = \int_u^\infty [ f(v) M(v) -  \frac{ 1 - M(v) }{\tau}  ]  \nu_\infty (v) \d v .$$
Let us write for short
\[ G(u) := \nu_\infty(u) (b(u) + \alpha)   = \gamma \exp \left( -\int_0^u \frac{f(y)}{b(y)+ \alpha } \d y \right),    \]
where $ \gamma = \gamma ( \alpha ) $ is the normalizing factor. 

Then
\[ G(u) M(u) = \int_u^\infty [ f(v) M(v) -  \frac{ 1 - M(v) }{\tau}  ]  \nu_\infty (v) \d v,  \]
and deriving with respect to $ u$ gives 
\begin{align*}
-\frac{f(u)}{b(u)+ \alpha } G(u) M(u) + G(u) M'(u) = - \left( f(u) M(u) - \frac{1-M(u)}{\tau} \right) \frac{G(u)}{b(u) + \alpha },
\end{align*}
such that
\[ M'(u) = \frac{1-M(u)}{\tau (b(u) + \alpha ) }. \]
The solutions of this ODE are of the form
\[ { M(v) = 1- C \exp\left(- \int_0^v \frac{1}{\tau (b(u) + \alpha) } \d u\right)  ,} \]
where $C$ is a constant to determine. To determine the constant, %
we use that by~\eqref{eq:ufm}, 
$$  U \int_0^\infty  f(v) M(v) \nu_\infty (v) \d v = \int_0^\infty  \nu_\infty ( v)  \frac{ 1 - M(v) }{\tau} \d v .$$
So
\begin{multline*}
      U \int_0^\infty  f(v) \left[1- C \exp\left(- \int_0^v \frac{1}{\tau (b(u) + \alpha ) } \d u\right) \right] \nu_\infty (v) \d v \\
      = \int_0^\infty  \nu_\infty ( v)  \frac{ C \exp\left(- \int_0^v \frac{1}{\tau (b(u) + \alpha) } \d u\right) }{\tau} \d v , 
\end{multline*}    
that is, 
\[  U \gamma(\alpha) = C \int_0^\infty  \exp\left(- \int_0^v \frac{1}{\tau (b(u) + \alpha) } \d u\right) (U f(v) + \frac{1}{\tau}) \nu_\infty(v) \d v . \]
As 
\[ \nu^\infty(v) = \frac{\gamma(\alpha)}{b(v)+ \alpha } \exp\left( - \int_0^v \frac{f(u)}{b(u)+ \alpha } \d  u \right), \]
we deduce, using the change of variables $ v = \varphi^\alpha_t (0), $
\[ {U = C - C(1-U) \int_0^\infty f(\varphi^\alpha_t(0)) \exp \left(-\int_0^t f(\varphi^\alpha_u(0)) \d u \right) e^{-\frac{t}{\tau}} \d t.}  \]
This last formula gives the formula of $C$.  
\end{proof}

\begin{remark}
In case $b(v) = \bar V - v ,  $ $\bar V \geq 0,$ 
we have
\[ { M(v) = 1 -C \left(1- \frac{v}{\bar V + \alpha } \right)^{1/\tau} }. \]
If $b(x) = b_0$, $b_0 \geq 0, $ we have
\[ {M(v) = 1 - C e^{-\frac{v}{\tau (b_0 + \alpha) } }}. \]
\end{remark}

\subsection{Revisiting \texorpdfstring{$\Theta_\alpha$}{Theta-alpha} in terms of Volterra integral equations}
Building upon the results obtained in Subsection~\ref{sec:32} in the time inhomogeneous case, in this section we study further the expressions appearing in the definition of $ \Theta_\alpha $ in~\eqref{eq:theta}. 
Recall (\ref{eq:Kt}) and define for all $x \in [0, 1],$ 
\[ R_t(x) : = \E_{(0, x)} X^\alpha_t f(V^\alpha_t) \mbox{ and } A_t(x) := K(t) \psi_t(x) , \]
where $K( t) = K^0 ( t) .$

With these notations, we have the following result. 

\begin{lemma}
We have the representation 
	\[ R_t(x) = R^1_t + R^2_t x, \]
	where the functions $R^1_t, R^2_t$ solve
\begin{align}\label{eq:R}
	R^1_t &= A^1_t + \int_0^t R^1_{t-s} K(s) \d s + \int_0^t R^2_{t-s} B^1_s K(s) \d s , \nonumber \\
	R^2_t &= A^2_t + \int_0^t R^2_{t-s} B^2_s K(s) \d s,
\end{align}
and where $A^1_t := K(t) (1- e^{-t/\tau})$ and $  A^2_t := K(t) e^{-t / \tau}.$
\end{lemma}

\begin{proof}
It follows analogously to Lemma~\ref{lem:33} that $(R_t(x))$ is the unique solution of the integral equation
	\begin{equation}\label{eq:integraleq1} R_t(x) = A_t (x) + \int_0^t{ R_{t-s}( (1-U) \psi_s(x)) K(s) \d s}. \end{equation} 
where  
\[ A_t(x) = A^1_t + A^2_t x . \]
The proof then follows analogously to the proof of Lemma~\ref{lem:33}. 
\end{proof}
In what follows we denote by $\widehat{K}(z)$ the Laplace transform of $K(t)$ given by
\[ \widehat{K}(z) := \int_0^\infty e^{-zt} K(t) \d t. \]

\begin{lemma}
	The Laplace transforms of $R^1$ and $R^2$ are given by
	\begin{align*}
		\widehat{R^1} (z) &= \frac{\widehat{K}(z) - \widehat{K}(z+1/\tau)}{[1-\widehat{K}(z)][1 - (1-U) \widehat{K}(z+1/\tau)]}, \\
		\widehat{R^2} (z) &= \frac{ \widehat{K}(z+1/\tau)}{ 1 - (1-U) \widehat{K}(z+1/\tau)}.
	\end{align*}
\end{lemma}
\begin{proof}
	First note that $B^2_s K(s) = (1-U) A^2_s$ and that $\widehat{A^2}(z) = \widehat{K}(z + 1/\tau)$. The equality for $\widehat{R^2}$ follows from the Volterra integral equation. Similarly, $B^1_s K(s) = (1-U) A^1_s$ and $\widehat{A^1}(z) = \widehat{K}(z) - \widehat{K}(z + 1/\tau)$. This implies the first equality. 
\end{proof}

\begin{remark}\label{rem:315}
{\bf 1)} For \[ S_t(v, x):= \E_{(v,x)} X^\alpha_t f(V^\alpha_t) \]
we have the representation
	\[ S_t(v, x) = K^v(t) \psi_t(x) + \int_0^t{ R_{t-s}((1-U) \psi_s(x)) K^v(s) \d s },  \]
	such that
	\[ S_t(v, x) = S^1_v(t) + S^2_v(t) x, \]
	with
	\begin{align*}
		S^1_v(t) &:= (1-e^{-t/\tau}) K^v(t) + \int_0^t K^v(s) R^1_{t-s} \d s + (1-U) \int_0^t K^v(s) (1-e^{-s/\tau}) R^2_{t-s} \d s, \\
		S^2_v(t) &:= e^{-t / \tau } K^v(t) + (1-U) \int_0^t K^v(s) e^{-s / \tau} R^2_{t-s} \d s. 
	\end{align*}
{\bf 2)} It follows from the explicit form of $ K^v $ and the fact that $ f'$ and $b'$ are Lipschitz continuous that $ \frac{\d }{\d v} S_t ( v, x) $ is Lipschitz continuous and bounded. \\
{\bf 3)}
 The Laplace transforms of $S^1_v$ and $S^2_v$ are given by
	\begin{align*}
		\widehat{S^1_v}(z) &= \widehat{K^v}(z) - \widehat{K^v}(z + 1/\tau) + \widehat{K^v}(z) \widehat{R^1}(z) + (1-U) [\widehat{K^v}(z) - \widehat{K^v}(z+1/\tau)] \widehat{R^2}(z), \\
		\widehat{S^2_v}(z) &= \widehat{K^v}(z+1/\tau) + (1-U) \widehat{K^v}(z+ 1/\tau) \widehat{R^2}(z).
	\end{align*}
\end{remark}

\subsection{Evaluating \texorpdfstring{$\widehat{\Theta}_\alpha(z)$}{the Laplace transform of Theta}}
Our goal is now to explicit the value of $\widehat{\Theta}_\alpha(z)$ in terms of the Laplace transforms of two more elementary functions. We consider
\begin{align*}\Xi_1(t) &:= J \int_{\R} \left[ \frac{\d }{\d v} K^v(t) \right] \nu_\infty(\d v). \\
	\Xi_2(t) &:= J \int_{\R } \left[ \frac{\d }{\d v}  K^v(t) \right] M(v) \nu_\infty(\d v).
\end{align*}
\begin{lemma}
	\label{lem:explicit_value_LTheta}
	We have
	\begin{align*}
		J \int \frac{\d }{\d v} S^v_1(z) \nu_\infty(\d v) &= \frac{\widehat{\Xi_1}(z) - \widehat{\Xi_1}(z + 1/\tau)}{1-(1-U)\widehat{K}(z+1/\tau)} + \frac{\widehat{\Xi_1}(z)( \widehat{K}(z) - \widehat{K}(z+1/\tau))}{(1-\widehat{K}(z))(1 - (1-U) \widehat{K}(z+1/\tau)) } ,\\
		J \int_0^\infty \left[ \frac{\d }{\d v} S^2_t(v) \right] \frac{\nu_\infty(\d v)}{1+ \tau U f(v)} &= \frac{\widehat{\Xi_2}(z+1/\tau)}{1-(1-U) \widehat{K}(z+1/\tau)}.
	\end{align*}
		Therefore,
		\[ \widehat{\Theta}_\alpha (z)  = \frac{1}{1-(1-U)\widehat{K}(z+\tfrac{1}{\tau})} \left\{ \widehat{\Xi_1}(z) - \widehat{\Xi_1}(z + \tfrac{1}{\tau}) + \widehat{\Xi_2}(z+\tfrac{1}{\tau}) +  \frac{\widehat{\Xi_1}(z)( \widehat{K}(z) - \widehat{K}(z+\tfrac{1}{\tau}))}{1-\widehat{K}(z)}  \right\}. \]
\end{lemma}

\subsection{Proof of Theorem~\ref{theo:2main}}
We start with some preliminary considerations. 
We recall that we start from some invariant probability measure $\mu_\infty(\d v, \d x)$ of the McKean-Vlasov equation~\eqref{eq:VXNL}.
Putting 
\[ \alpha = J \int_{\R  \times [0,1]} x f(v) \mu_\infty(\d v, \d x),  \]
$\mu_\infty = \mu_\infty^\alpha$ is also the unique invariant probability measure of the linear process $(V^\alpha_t, X^\alpha_t)$.
Consider $\nu_\infty(\d v) = \mu_\infty(\d v, [0, 1])$ the first marginal of $\mu_\infty$. We recall that this first marginal is known explicitly.
Then we can rewrite, with the above expressions, 
\begin{align*} \Theta_\alpha(t) &:= J \int_{\R  \times [0,1]} \frac{\d }{\d v} \E_{(v,x)}[X^{\alpha}_t f(V^\alpha_t)] \mu_\infty(\d v, \d x) \\
&=  J \int_{\R  \times [0,1]} \frac{\d }{\d v} S_t(v, x) \mu_\infty(\d v, \d x) \\
&= J \int_\R \frac{\d }{\d v} S^1_t(v) \nu_\infty (\d v) + J \int_\R \left[ \frac{\d }{\d v} S^2_t(v) \right] M(v) \nu_\infty(\d v).
\end{align*}
Here we have used Proposition~\ref{prop:invm} to obtain the last identity.

We now turn to the proof of Theorem~\ref{theo:2main}. We first show the following perturbation result which is a Trotter-Kato type formula. It relates the difference of two semigroups to the one of the associated generators. 

\begin{proposition}\label{prop:TK}
Let $ g \in C^1 ( \R \times [0, 1 ]; \R )  $ be bounded and $ k \in C( \R_+; \R) . $ Suppose that $ \cL ((V_0^{ \alpha + k }, X_0^{ \alpha + k })  ) =  \cL ((V_0^{ \alpha  }, X_0^{ \alpha  })  ) = \mu .$ Then we have that 
\begin{multline}
   \E  g (V_t^{ \alpha + k, \mu }, X_t^{ \alpha + k , \mu }) -  \E  g (V_t^{ \alpha , \mu  }, X_t^{ \alpha , \mu  }) \\
   =\int_0^t \int_\R \left[ \frac{\d }{\d v } \E_{(v, x) } g( (V^\alpha_{ t-s } , X^\alpha_{t-s} ))\right] k_s \cL  (V_s^{ \alpha + k, \mu  } , X_s^{ \alpha + k }, \mu  ) (\d v, \d x) ds .
\end{multline}
\end{proposition}
\begin{proof}
The proof is along the lines of the proof of Prop. 3.13 of \cite{MNA}, and so we only sketch the main arguments. In what follows, let us write $ z := ( v, x) .$ Since we work with a fixed initial distribution $\mu, $ to ease the reading, we shall drop the superscript $ \mu $ when denoting our process. We put for any $ 0 \le s \le t, $ $ \Phi (s, z) := \E_z (g( V_{t-s}^{ \alpha  }, X_{t-s}^{ \alpha  }))  .  $ Using the arguments of Section~\ref{sec:32}, it follows that $ \Phi \in C^1_b ( [ 0, t ] \times \R \times [0, 1 ]; \R ).$ So,
$$ \frac{\partial}{\partial s} \Phi (s, z) = - A^\alpha \phi ( s, z) , $$
where 
$$ A^\alpha g ( z) = \frac{ \partial g}{\partial v} (z) [ b( v ) + \alpha ] + \frac{ \partial g}{\partial x} (z) \frac{ 1 -x }{\tau} + f ( v) [ g ( 0, (1- U) x) - g( v, x) ] $$
is the generator of $ (V^\alpha, X^\alpha ). $ Notice that the time dependent generator $ A^{ \alpha + k}_s $ of $(V^{ \alpha + k }, X^{ \alpha + k } )$ satisfies 
$$(A^{ \alpha + k}_s - A^{ \alpha } ) g ( z) = \frac{ \partial g}{\partial v} (z) k_s.  $$
Hence, applying first Ito's formula and then replacing $ \Phi ( u, z) $ by its definition,  
\begin{multline*} \E \Phi ( s, (V_s^{ \alpha + k }, X_s^{ \alpha + k } ) ) = \E \Phi ( 0, (V_0^{ \alpha + k }, X_0^{ \alpha + k } ) )  + \E \int_0^s \frac{\partial }{ \partial v} \Phi ( u, z)_{ | z=(V_u^{ \alpha + k }, X_u^{ \alpha + k } )} k_s \d s \\
= \E \Phi ( 0, (V_0^{ \alpha + k }, X_0^{ \alpha + k } ) )  +  \int_0^s  \int_{\R \times [0, 1  ]} \left[ \frac{\partial }{ \partial v} \E_{(v, x) } g( V_{t-u}^\alpha, X_{t-u}^\alpha)\right]  k_u \cL  (V_u^{ \alpha + k } , X_u^{ \alpha + k } ) (\d v, \d x).
\end{multline*}
Letting $ s \to t$ and observing that $ \Phi( t, z ) = g(z) = g ( v, x) , $ the assertion then follows. 
\end{proof}

Together with Theorem~\ref{prop:erglin}, we deduce from the above the following control.
\begin{lemma}\label{lem:318}
We have that
\begin{equation}\label{eq:323}
d_{BL}( \cL   (V_t^{ \alpha + k, \mu }, X_t^{ \alpha + k, \mu }), \cL (V_t^{ \alpha , \mu }, X_t^{ \alpha , \mu  })) \le C_* \int_0^t e^{ - \lambda_* ( t-s) } | k_s | \d s .   
\end{equation}
\end{lemma}

The proof of this result is straightforward, using that the class of differentiable and bounded functions is dense in the class of bounded $ Lip_1$ functions. Details are omitted and can be found in the proof of Corollary 3.14 in \cite{MNA}. 

\subsection*{Control of the non-linear interactions.} A main step of our proof is the study of the perturbation errors 
$$ \varphi_t^\mu := J \E ( f( V_t^{ \alpha, \mu } ) X_t^{ \alpha , \mu } ) - \alpha  $$
for the linearized, Markovian version of the process, 
and of
$$ k_t^\mu := J \E ( f( \bar V_t^\mu ) \bar X_t^\mu ) - \alpha ,  $$
for the true non-linear version of the process. Since $ ( \bar V^\mu, \bar X^\mu ) = ( V^{ \alpha + k^\mu, \mu }, X^{\alpha + k^\mu , \mu }), $ we have that 
$$ d_{BL} (\mu_t, \mu_\infty ) = d_{BL} ( \cL  ( V_t^{ \alpha + k^\mu, \mu  }, X_t^{\alpha + k^\mu }, \mu  ) ,\cL  ( V_t^{ \alpha , \mu  }, X_t^{\alpha , \mu } ) ) .$$
By (\ref{eq:323}) and the perturbation result of Proposition~\ref{prop:TK}, it is therefore crucial to control $ k^\mu_t.$ However, it difficult to deal directly with $ k^\mu_t,$ while we know how to deal with $ \varphi_t^\mu .$ These steps have been well established in previous work, see \cite{cormier2023stability}, and in what follows we directly adapt the arguments of \cite{MNA} to our present framework.

Analogously to Proposition 3.15 of \cite{MNA}, we have for all $ t \le T$ and for a constant $ C_T $ depending only on $T,$ 
\begin{equation}\label{eq:ktmu}
 | k_t^\mu  - \varphi_t^\mu - \int_0^t \Theta_\alpha (t-s) k_s^\mu \d s | \le C_T    d_{BL}( \mu, \mu_\infty )^2. 
\end{equation}
It is possible to resolve the above equation such that we only have to deal with the known object $ \varphi_t^\mu.$ This is done by introducing $ \Omega_\alpha ( t) ,$ the solution of the Volterra integral equation 
$$ \Omega_\alpha ( t) = \Theta_\alpha (t) + \int_0^t \Omega_\alpha ( t-s) \Theta_\alpha ( s) \d s , t \geq 0. $$
We have the following first result on $ \Omega_\alpha (t).$
\begin{lemma}
For all $ \lambda < \lambda'_\alpha , $ we have $ \sup_{ t \geq 0} | \Omega_\alpha ( t) | e^{\lambda t } < \infty . $    
\end{lemma}

\begin{proof}
The proof follows along the lines of the proof of Lemma 3.12 in \cite{MNA}, observing that  $K ( t) =  e^{ \lambda t } \Theta_\alpha  (t) $ is integrable and that $ \hat K ( z) \neq 0 $ for all $ z $ having positive real part.      
\end{proof}
Iterating the estimate (\ref{eq:ktmu}), we then obtain the following key estimate 
\begin{equation}\label{eq:ktmu2}
 | k_t^\mu  - \varphi_t^\mu - \int_0^t \Omega_\alpha (t-s) \varphi_s^\mu \d s | \le C_T    d_{BL}( \mu, \mu_\infty )^2 . 
\end{equation}

\begin{lemma}
For any $ \lambda \in ( 0, \lambda_\alpha ' ), $ there exists a constant $ C_\lambda $ such that for all $ T >0 $ there exists $C_T > 0 $ with the following property. For all $ \mu \in \cP ( \R \times [0, 1 ] ), $ for all $ 0 \le t \le T, $ 
$$ d_{BL} ( \cL  ( \bar V_t^{ \mu}, \bar X_t^{ \mu } ), \mu_\infty ) \le C_\lambda e^{ - \lambda t } d_{BL}  ( \mu, \mu_\infty ) + C_T d_{BL} ( \mu, \mu_\infty)^2. $$
\end{lemma}

\begin{proof}
We use that $ ( \bar V_t^{ \mu}, \bar X_t^{ \mu } ) = ( V_t^{ \alpha + k^\mu , \mu }, X_t^{\alpha + k^\mu , \mu  } )$ to obtain 
$$ d_{BL} ( \cL  ( \bar V_t^{ \mu}, \bar X_t^{ \mu } ), \mu_\infty )  \le d_{BL} ( \cL  ( V_t^{ \alpha + k^\mu , \mu  }, X_t^{\alpha + k^\mu , \mu } ) , \cL ( V_t^{ \alpha , \mu  }, X_t^{\alpha , \mu   } ) ) + d_{BL} ( \cL ( V_t^{ \alpha , \mu  }, X_t^{\alpha , \mu  } ) , \mu_\infty^\alpha ) .  $$
By Lemma~\ref{lem:318}, and using the above estimates,  
\begin{multline*} d_{BL} ( \cL  ( V_t^{ \alpha + k^\mu , \mu }, X_t^{\alpha + k^\mu , \mu } ) ,\cL  ( V_t^{ \alpha , \mu  }, X_t^{\alpha , \mu  } ) ) \le C_* \int_0^t e^{ - \lambda_* ( t-s) } | k_s^\mu |  \d s \\
\le C_* \int_0^t e^{ - \lambda_* (t-s) } \left( | \varphi_s^\mu | + \int_0^s | \Omega_\alpha ( s-u  ) | | \varphi_u^\mu | \d u \right) \d s + C_T d_{ BL} ( \mu, \mu_\infty)^2.   
\end{multline*}
Since $ \R \times [0, 1 ] \ni (v, x) \mapsto J f(v) x $ is Lipschitz and bounded, we have by Theorem~\ref{prop:erglin}, 
$$ | \varphi_t^\mu | \le C_* J [ \|f\|_{Lip} + \| f \|_\infty ]e^{ - \lambda_ * t } d_{BL} ( \mu, \mu_\infty). $$
So, for a convenient constant, 
$$ C_* \int_0^t e^{ - \lambda_* (t-s) }  | \varphi_s^\mu | \d s \le C e^{- \lambda_* t }d_{BL} ( \mu, \mu_\infty) .  $$
Next we fix some $ \lambda < \lambda_\alpha' < \lambda_* $ such that, for some constant $C,$ $ | \Omega_\alpha ( s  ) | \le C e^{ - \lambda s } ,$ for all $s.$ This implies that 
$$\int_0^t e^{ - \lambda_* (t-s) }  \int_0^s | \Omega_\alpha ( s-u  ) | | \varphi_u^\mu | \d u \d s \le C_{\lambda} e^{ - \lambda t }  d_{BL} ( \mu, \mu_\infty) , $$
whence the assertion. 
\end{proof}

The proof of Theorem~\ref{theo:2main} is now straightforward and follows from the above result; see \cite{cormier2023stability} for the details.

\section{Numerical methods and examples}\label{sec:num}
\subsection{Numerical implementation}
In this section we explain how to find numerically the invariant distributions and compute their local stability using our main results. 
An implementation of the method presented below (together with a javascript interface) can be found on the \href{https://quentincormier.org/2Dspiking.html}{following web page}.

In what follows, we assume the parameters $b, f, \tau$ and $U$ to be fixed. The goal of our approach is to restate our criteria in terms of simple  ODE's that can be easily implemented. Empirically, we have observed that the ODE approach described below is numerically more stable than the naive approach consisting of directly implementing the formulas of Lemma~\ref{lem:explicit_value_LTheta}.
First, following Proposition~\ref{prop:atleastone}, there is a one-to-one correspondence between the invariant distributions of \eqref{eq:VXNL} and the solution of the scalar equation
\[ J = \frac{\alpha }{ \int_{\R \times [0, 1]} x f(v) \mu^\infty_\alpha(\d v , \d x)} = \frac{\alpha }{ \int_{\R} f(v) M(v) \nu^\infty_\alpha(v) \d v}. \]
To compute the right hand side, we use:
\begin{align*}
&	\frac{1}{\gamma(\alpha)} = \int_0^\infty  \exp\left( -\int_0^t f(\varphi^\alpha_u (0)  ) \d u \right) \d t, \\  
&            C^\alpha = U \left( 1 -  (1-U) \int_0^\infty f(\varphi^\alpha_t(0)) \exp \left(-\int_0^t f(\varphi^\alpha_u(0)) \d u \right) e^{-\frac{t}{\tau}} \d t \right)^{-1}, \\
&	\int_{\R} f(v) M(v) \nu^\infty_\alpha(v) \d v = \gamma(\alpha) \left[ 1 - C^\alpha \int_0^\infty K_\alpha(t) e^{-\frac{t}{\tau}} \d t \right]  , 
\end{align*}
Therefore, we solve the following ODE:
\[
  \mathbf{x}(0) = \begin{bmatrix} 
        0 \\ 1 \\ 0 \\ 0  
    \end{bmatrix},  \quad 
  \mathbf{x}(t) = \begin{bmatrix} 
        x_1(t) \\ x_2(t) \\ x_3(t) \\ x_4(t) 
    \end{bmatrix}, \quad     \frac{d}{dt} \mathbf{x}(t) = \begin{bmatrix} 
        b(x_1(t)) + \alpha \\ 
        - f(x_1(t)) x_2(t)\\ 
        x_2(t) \\ 
	f(x_1(t)) x_2(t) e^{-t/\tau}
    \end{bmatrix} ,
\]
so that
\[ x_1(t) = \varphi^\alpha_t(0), \quad x_2(t) = H_\alpha(t), \quad x_3(t) = \int_0^t H_\alpha(s) \d s, \quad x_4(t) =  \int_0^t f(\varphi^\alpha_s(0)) H_\alpha(s) e^{-\frac{s}{\tau}} \d s.  \]
Finally, we compute the solution of this ODE for $t = T$ large enough such that
\[  \int_{\R} f(v) M(v) \nu^\infty_\alpha(v) \d v \approx \frac{1}{x_3(T)} \left[ 1 - U \frac{x_4(T)}{ 1-(1-U) x_4(T)} \right].  \]
We now explain how to decide whether a given invariant distribution is stable or not. To compute $\widehat{\Theta}_\alpha (z)$, we rely on Lemma~\ref{lem:explicit_value_LTheta}:
	\[ \widehat{\Theta}_\alpha (z)  = \frac{1}{1-(1-U)\widehat{K}(z+\tfrac{1}{\tau})} \left\{ \widehat{\Xi_1}(z) - \widehat{\Xi_1}(z + \tfrac{1}{\tau}) + \widehat{\Xi_2}(z+\tfrac{1}{\tau}) +  \frac{\widehat{\Xi_1}(z)( \widehat{K}(z) - \widehat{K}(z+\tfrac{1}{\tau}))}{1-\widehat{K}(z)}  \right\}. \]
We introduce
\[ \Psi_1(t) = - J \int_{\R} \left[ \frac{d}{d v} H^v(t) \right] \nu_\infty(\d v) = J \gamma(\alpha) \int_0^\infty H_\alpha(t+u) \frac{f(\varphi^\alpha_{t+u}(0)) - f(\varphi^\alpha_u(0))}{b(\varphi^\alpha_u(0) + \alpha} \d u. \]
We first explain the strategy to compute efficiently (and accurately) $\widehat{\Psi}_1$.
We first note that
\[ \widehat{\Psi}_1(z) =  J \gamma(\alpha)  \int_0^\infty \frac{e^{zu}}{b(\varphi^\alpha_u(0)) + \alpha} \left[ \int_u^\infty e^{-zs} K_\alpha(s) \d s - f(\varphi^\alpha_u(0)) \int_u^\infty e^{-zs} H_\alpha(s) \d s \right] \d u. \]
Integrating by parts, we find that
\begin{multline*} \widehat{\Psi_1}(z) = J \gamma(\alpha)  \int_0^\infty \left[ \int_0^u \frac{e^{z \theta}}{b(\varphi^\alpha_\theta) + \alpha} \d \theta \right] e^{-z u }  K_\alpha(u) \d u \\
-  J \gamma(\alpha)  \int_0^\infty \int_0^u \left[ \frac{f(\varphi^\alpha_\theta) e^{z \theta}}{b(\varphi^\alpha_\theta) + \alpha} \d \theta \right] e^{-z u} H_\alpha(u) \d u.   
\end{multline*} 
To get rid of the double integrals, we now define
\[ A_1(u) = \left[ \int_0^u \frac{e^{z \theta}}{b(\varphi^\alpha_\theta) + \alpha} \d \theta \right] e^{-zu} (b(\varphi^\alpha_u(0)) + \alpha), \]
and we note that
\[ \frac{\d }{\d u} A_1(u) = 1 + (b'(\varphi^\alpha_u) - z) A_1(u).  \]
To deal with the other double integral appearing in the expression of $ \widehat{\Psi}_1(z),$ we define
\[ B_1(u) = \left[ \int_0^u \frac{e^{z \theta} f(\varphi^\alpha_\theta (0))}{b(\varphi^\alpha_\theta) + \alpha} \d \theta \right] e^{-zu} (b(\varphi^\alpha_u(0)) + \alpha), \]
so that
\[ \frac{\d }{\d u} B_1(u) = f(\varphi^\alpha_u(0)) + (b'(\varphi^\alpha_u) - z) B_1(u).  \]
Altogether,
\[  \widehat{\Psi}_1(z) = J \gamma(\alpha)  \int_0^\infty \frac{A_1(u)}{b(\varphi^\alpha_u(0)) + \alpha} K_\alpha(u) \d u -  J \gamma(\alpha)  \int_0^\infty  \frac{B_1(u)}{b(\varphi^\alpha_u(0)) + \alpha} H_\alpha(u) \d u.   \]
We similarly introduce
\[ \Psi_2(t) = - J \int_{\R} \left[ \frac{d}{d v} H^v(t) \right] M(v) \nu_\infty(\d v). \] 
Substituting $M(v)$ by its explicit expression, we find that:
\[ \Psi_2(t) = \Psi_1(t) - C_\alpha J \gamma(\alpha) \int_{\R+} H_\alpha(t+u) e^{-\frac{u}{\tau}} \frac{ f(\varphi^\alpha_{t+u}(0)) - f(\varphi^\alpha_u(0))}{b(\varphi^\alpha_u(0)) + \alpha} \d u.   \]
We define then $A_2(t)$ and $B_2(t)$ to be the solution of the following ODE:
\[ \frac{\d}{\d t} A_2(t) = e^{-t/\tau} + (b'(\varphi^\alpha_t(0)) - z) A_2(t), \quad A_2(0) = 0. \]
and
\[ \frac{\d}{\d t} B_2(t) = e^{-t/\tau}f(\varphi^\alpha_t(0)) + (b'(\varphi^\alpha_t(0)) - z) B_2(t), \quad B_2(0) = 0, \]
so that
\[ \widehat{\Psi}_2(z) = \widehat{\Psi}_1(z) - C_\alpha J \gamma(\alpha) \int_0^\infty \frac{A_2(u) K_\alpha(u) - B_2(u) H_\alpha(u)}{b(\varphi^\alpha_u) + \alpha} \d u. \]
Altogether, in order to compute $\widehat{\Psi}_1(z)$ and $\widehat{\Psi}_2(z)$, we solve the following ODE:
\[
  \mathbf{y}(0) = \begin{bmatrix} 
        0 \\ 1 \\ 0 \\ 0 \\ 0 \\ 0 \\ 0 \\ 0 
    \end{bmatrix},  \quad 
  \mathbf{y}(t) = \begin{bmatrix} 
        y_1(t) \\ y_2(t) \\ y_3(t) \\ y_4(t) \\ y_5(t) \\ y_6(t) \\ y_7(t) \\ y_8(t) 
    \end{bmatrix}, \quad     \frac{d}{dt} \mathbf{y}(t) = \begin{bmatrix} 
        b(y_1(t)) + \alpha \\ 
        - f(y_1(t)) y_2(t)\\ 
	1+ [b'(y_1(t)) - z] y_3(t) \\ 
	f(y_1(t)) + [b'(y_1(t)) - z] y_4(t) \\
	  y_2(t) \left[ \frac{y_3(t) f(y_1(t))}{b(y_1(t)) + \alpha} - \frac{y_4(t)}{b(y_1(t)) +\alpha} \right] \\
	  e^{-t/\tau}+ [b'(y_1(t)) - z] y_6(t) \\ 
	  f(y_1(t)) e^{-t/\tau} + [b'(y_1(t)) - z] y_7(t) \\
	  y_2(t) \left[ \frac{y_6(t) f(y_1(t))}{b(y_1(t)) + \alpha} - \frac{y_7(t)}{b(y_1(t)) +\alpha} \right] 
    \end{bmatrix} ,
\]
so that 
\[ y_1(t) = \varphi^\alpha_t(0), \quad y_2(t) = H^\alpha_t, \quad y_3(t) = A_1(t), \quad y_4(t) = B_1(t), \quad y_6(t) = A_2(t), \quad y_7(t) = B_2(t). \]
In addition, for $T$ large enough, it holds that
\[ \widehat{\Psi}_1(z) \approx J \gamma(\alpha)  y_5(T), \quad \widehat{\Psi}_2(z) \approx  J \gamma(\alpha)  y_5(T) - C_\alpha J \gamma(\alpha) y_8(T).  \]
The derivative of $\widehat{\Psi}_1(z)$ and $\widehat{\Psi}_2(z)$ with respect to $z$ are computed similarly, by solving an ODE. Finally, $\widehat{\Theta}_\alpha (z)$ and its derivative are computed using the formula of Lemma~\ref{lem:explicit_value_LTheta}. Finally, consider a clockwise contour $\gamma$ which consists of a line from -$iR$ to $i R$ on the imaginary axis and a semicircle on the right half plane.
The number of solutions of the equation $1-\widehat{\Theta}_\alpha (z) =0$ is given by:
\[ \frac{1}{2 \pi i} \int_{\gamma_R} \frac{\widehat{\Theta}_\alpha '(z)}{\widehat{\Theta}_\alpha (z) - 1} \d z. \]
As $R \rightarrow \infty$, this quantity converges to:
\[ \frac{1}{\pi} \int_0^\infty \frac{\widehat{\Theta}'_\alpha(iy)}{\widehat{\Theta}_\alpha(iy) - 1} \d y. \]
This quantity is an integer. We compute it using a trapezoidal method.
When this quantity is equal to zero, the invariant distribution is stable. 

\section{Examples}\label{sec:5}
In this section we give two examples where our criterion can be successfully assessed.

\subsection{A first example}
We first discuss an example where spiking only occurs when the potential is above a certain threshold. We choose
\begin{equation}\label{eq:exbeta}
f(x) = \begin{cases} 0 \quad \text {if } x \leq 1 \\ 
\frac{1}{\beta} \quad \text {if } x > 1 , \end{cases} ,\; 
b(x) = m - x  ,
\end{equation}
for some parameters $m> 1$ and $\beta > 0$.
Let
\[
\omega
= \log\!\left( \frac{m + \alpha}{m + \alpha - 1} \right), \quad \delta = \frac{\alpha}{m + \alpha - 1}.
\]
We find that
\[
\widehat{H}_\alpha(z) =  \frac{1 - e^{- \omega z }}{z}
+ \frac{e^{- \omega z}}{\,z + 1/\beta\,}, \quad \widehat{K}_\alpha(z) = 1 - z \widehat{H}_\alpha(z).
\]
Therefore, 
\[ \widehat{K}_\alpha(z) = \frac{e^{-\omega z}}{1+\beta z}\]
and
\begin{equation}\label{eq:psi1_explicit}
\widehat{\Psi}_1(z) = \frac{1 - (1-U)\widehat{K}_\alpha(1/\tau)}{1 - \widehat{K}_\alpha(1/\tau)} \delta \frac{1 - e^{-\omega(z+1)}}{(1+\beta z) (z+1)}. 
\end{equation}
Moreover, the constant $C^\alpha$ is given by
 \[ C^\alpha = \frac{U}{1-(1-U) \widehat{K}_\alpha(\frac{1}{\tau})}, \quad \widehat{K}_\alpha(\frac{1}{\tau}) = \frac{\tau}{\tau + \beta} e^{-\frac{\omega}{\tau}}.\]
Finally, 
\begin{equation}
\widehat{\Psi}_2 (z) = \widehat{\Psi}_1 (z)  - \frac{C^\alpha \delta}{1 + \beta z} \frac{1 - (1-U)\widehat{K}_\alpha(1/\tau)}{1 - \widehat{K}_\alpha(1/\tau)} \frac{e^{-\frac{\omega}{\tau}} - e^{-\omega(z+1)} }{ z+1 - \frac{1}{\tau}}.
\end{equation}
After some simplifications, we find that the condition $\widehat{\Theta}_\alpha (z) = 1$ is equivalent to \begin{equation}
1-(1-U)\widehat{K}(z+1/\tau) = \widehat{\Psi}_1(z) \left[ z \frac{1-\widehat{K}(z+1/\tau)}{1-\widehat{K}(z)} - (z+1/\tau) C^\alpha \frac{\widehat{K}(z+1/\tau)}{\widehat{K}(z)} \right],
\end{equation}
which is an explicit equation where $\widehat{K}(z) = \frac{e^{-\omega z}}{1+\beta z}$ and $\widehat{\Psi}_1(z)$ is given by \eqref{eq:psi1_explicit}.
Using these explicit formulas, we can study the effect of the parameters $U$ and $\tau$ on the stability of the stationary solutions. The result are gathered in Figure~\ref{fig:explicit-stab-step}. We find that increasing $U$ and decreasing $\tau$ tends, in the example, to stabilize the invariant distribution and to prevent the emergence of oscillations in the system. 
\begin{figure}[!ht]
    \centering
    \includegraphics[width=0.8\textwidth]{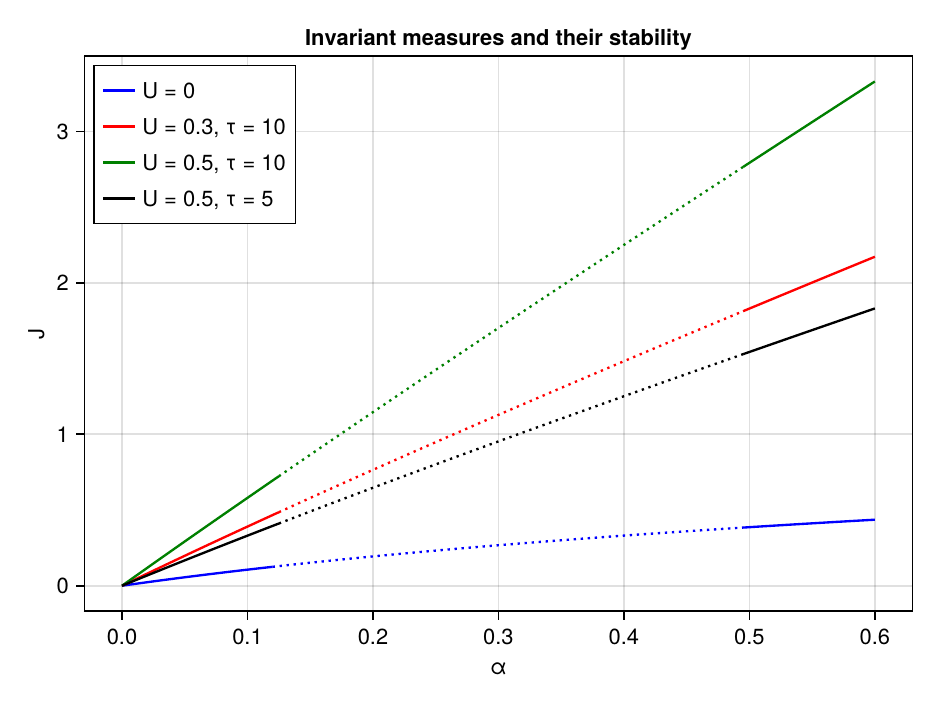}
    \caption{Simulation of the model (\ref{eq:exbeta}) for $\beta = 0.08$, $b(x) = 3/2 - x$.
    We see that in the range $\alpha \in[0,0.75]$, there is uniqueness of the invariant probability measure. Stability does not hold in an interval, suggesting Hopf bifurcations. We study the effect of the parameters $U$ and $\tau$ on the stability of the invariant distribution. We see that increasing $U$ or decreasing $\tau$ tends to increase the range of stability of the invariant distribution.  
    }
    \label{fig:explicit-stab-step}
\end{figure}

\subsection{Self Organized Bistability}
In what follows, we study the dynamics of the model with the following parameters:
\[ b(v) = b_0 - v, \quad f(v) = \max(0,v)^2, \quad b_0 = 0.05. \] 
We also choose $J = 6$, $U = 0.3$ and $\tau = 30$. With this set of parameters, the McKean-Vlasov equation~\eqref{eq:VXNLintro} has a unique invariant distribution, which is locally stable, see Figure~\ref{fig:SOC-nyquist}. However, when we simulate the model, an interesting oscillatory behavior appears, as shown in Figure~\ref{fig:raster-SOC}. The simulation is run on $[0, T]$, with $T = 600$. We observe that each neuron fires on average 41 times, with a relatively small standard deviation of $3.5$.
This is to be compared to the period of the oscillations, approximately equal to $75$. Therefore, each neuron spikes approximately 5 times per period of the macroscopic oscillations! This behavior is very different from the periodic solutions found in the 1D model, where typically one neuron is firing one time per period of the collective oscillations.

The observed behavior is well understood following \cite{PhysRevResearch.2.013318}, where a toy model with a 2D slow-fast ODE is shown to exhibit a ``self-organized'' bistable behavior.

With our choice of parameters, we are close to a slow-fast dynamics: the dynamics of $(\bar V_t)$ is much faster than the dynamics of $(\bar X_t)$.
We first study the invariant distributions (and their stability) of the corresponding one-dimensional model; that is, we assume that $(\bar X_t)$ is constant, and define
\[ J_{\text{eff}} = J \E (\bar X_t) . \]
Therefore, we consider the dynamics of the one-dimensional model:
\[ \d \tilde{V}_t = b(\tilde{V}_t) \d t + J_{\text{eff}} \E f(\tilde{V}_t) \d t - \int_0^t \int_{\R} \tilde{V}_{s-} \indic{z \leq f(\tilde{V}_{s-})} N(\d s, \d z). \] 
We show that depending on the value of $J_\text{eff}$, this equation has a bistable behavior, see Figure~\ref{fig:SOC-1D-inv-measure}:
\begin{itemize}
	\item When $J < J_0 \approx 2$ or $J > J_1 \approx 4.6$, $(\tilde{V}_t)$ has a unique invariant distribution, which is unstable. 
	\item When $J \in (J_0, J_1)$, there are 3 invariant distributions, two of them are stable. One corresponds to a very small activity of the network, the other to a large activity. 
\end{itemize}
Finally, we plot the value of $J_{\text{eff}}(t) = J \E (\bar X_t)$, see Figure~\ref{fig:SOC-Jeff}, showing that the effective $J$ indeed explores the two critical values $J_0$ and $J_1$.
To conclude, with this set of parameters, the model~\eqref{eq:VXNLintro} is therefore a concrete/realistic implementation of the mechanism of self-organized bistability discussed in~\cite{PhysRevResearch.2.013318}.

\begin{figure}[!ht]
    \centering
    \begin{subfigure}[b]{0.48\textwidth}
        \centering
        \includegraphics[width=\textwidth]{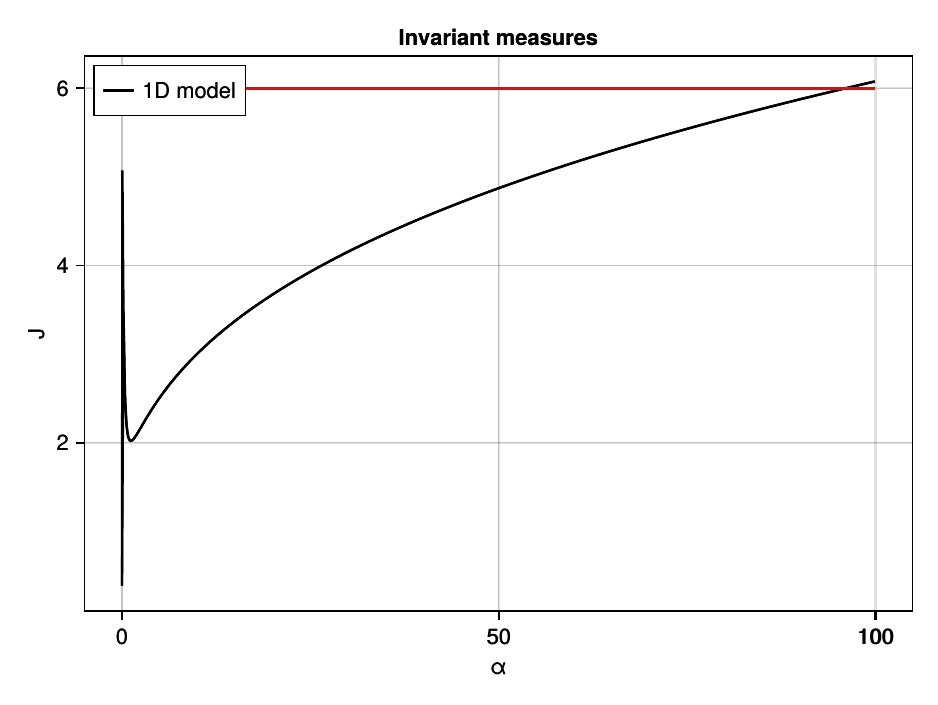}
        \caption{Full range.}
        \label{fig:SOC-1D-inv-measure}
    \end{subfigure}
    \hfill %
    \begin{subfigure}[b]{0.48\textwidth}
        \centering
        \includegraphics[width=\textwidth]{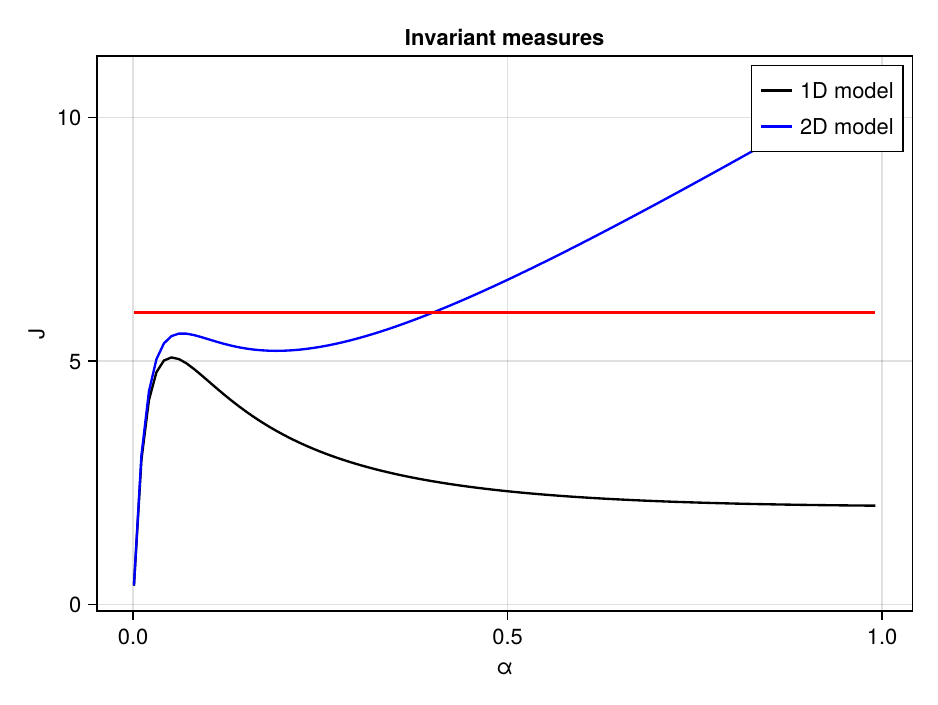}
        \caption{Zoomed for $\alpha \in [0, 1]$.}
        \label{fig:SOC-2D-inv-measure}
    \end{subfigure}

    \caption{Number of invariant distributions of the 1D equivalent model and the 2D model.}
    \label{fig:combined-SOC-measure}
\end{figure}

\begin{figure}[!ht]
    \centering
    \includegraphics[width=0.8\textwidth]{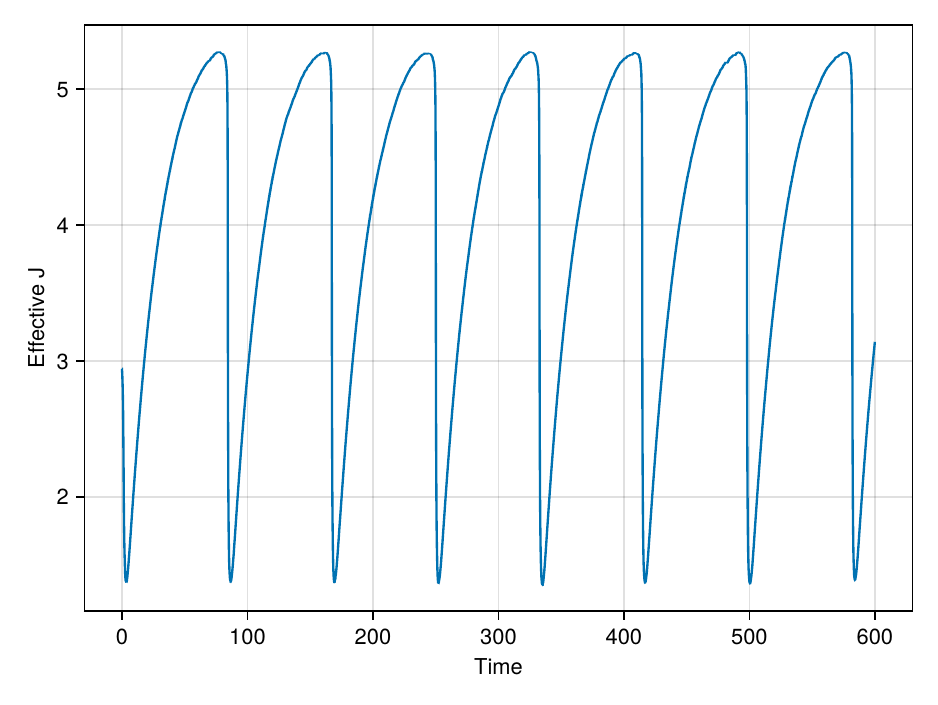}
    \caption{ Plot of $J_{\text{eff}} = J \E X_t$ as a function of time. We observe that the dynamics explores the two critical parameters of the 1D  model, which are $J_0 = 2$ and $J_1 = 4.6$. For $J_{\text{eff}} \in (J_0, J_1)$, the dynamics of the 1D-model is bistable (two stable invariant distributions).  }
    \label{fig:SOC-Jeff}
\end{figure}

\section{Appendix: heuristic derivation of the spectral condition}
\label{app:spectral_heuristic}

In this appendix, we provide an informal, partial differential equation (PDE)-based derivation of the spectral condition $\widehat{\Theta}_\alpha(z) = 1$ determining the local stability of an invariant probability measure. The aim here is to build rapid physical intuition by relying on the linearization of the non-linear Fokker-Planck equation\footnote{We use the term Fokker-Planck equation here to guide the intuition, the PDE is actually a hyperbolic equation of transport type; see \cite{Carrillo2025} for a comprehensive review of PDEs in neuroscience.} around the equilibrium. 

Let $p(t, v, x)$ denote the probability density function associated to the law of $(V_t, X_t)$. The macroscopic dynamics solve (in a weak sense) the non-linear Fokker-Planck equation:
\[
    \partial_t p = -\partial_v \left[ \left(b(v) + J\iint x' f(v') p(t,  \d v', \d x')  \right) p \right] - \partial_x \left[ \frac{1-x}{\tau} p \right] + \mathcal{S} p,
\]
where the jump operator $\mathcal{S}$ is defined weakly for any suitable test function $g$ as:
\[
    \iint g(v,x) (\mathcal{S} p)(v,x) \d v \d x = \iint \left[g(0, (1-U)x) - g(v, x)\right] f(v) p(t, v, x) \d v \d x.
\]
To investigate local stability, we linearize this PDE around a stationary equilibrium state $\mu_\infty$. We freeze the non-linear interaction term by introducing the constant $\alpha = J\iint x f(v) \mu_\infty(\d v, \d x)$ and consider a small perturbation $\phi(t, v, x)$ such that $p(t, v, x) = \mu_\infty(v, x) + \phi(t, v, x)$. Let $\mathcal{L}_\alpha^*$ be the forward Fokker-Planck operator (the adjoint of the infinitesimal generator) for the linear process $(V_t^\alpha, X_t^\alpha)$:
\[
    \mathcal{L}_\alpha^* p = -\partial_v \left[ (b(v)+\alpha) p \right] - \partial_x \left[ \frac{1-x}{\tau} p \right] + \mathcal{S} p.
\]
Keeping only the first-order terms in the expansion of the non-linear PDE, the perturbation $\phi$ evolves according to the linearized equation:
\[
    \partial_t \phi = \mathcal{L}_\alpha^* \phi - J \left( \iint x' f(v') \phi(t, \d v', \d x') \right) \partial_v \mu_\infty.
\]
We analyze the spectrum of this linear PDE by considering normal modes of the form $\phi(t, v, x) = e^{zt}\varphi(v, x)$. Injecting this into the linearized equation yields the eigenvalue problem:
\[
    (z I - \mathcal{L}_\alpha^*) \phi = - J \left( \iint x' f(v') \phi(t,\d v', \d x') \right) \partial_v \mu_\infty.
\]
Let $\mathcal{T}_\alpha^*(t)$ be the semigroup generated by $\mathcal{L}_\alpha^*$. We define the resolvent operator $R(z) = (z I - \mathcal{L}_\alpha^*)^{-1}$ via its Laplace transform $R(z) = \int_0^\infty e^{-z t} \mathcal{T}_\alpha^*(t) \d t$. Applying the resolvent to both sides, we obtain an expression for the eigenmode $\phi$:
\begin{equation}
    \label{eq:phi_resolvent}
    \phi = - J \left( \iint x' f(v') \phi(t,\d v', \d x') \right) R(z) [\partial_v \mu_\infty].
\end{equation}

To find a closed condition for the eigenvalues $z$, we project this equation along the interaction observable. Multiplying both sides of \eqref{eq:phi_resolvent} by $x f(v)$ and integrating over the state space gives:
\begin{equation}
    \iint x f(v) \phi(t,\d v, \d x) = - J \left( \iint x f(v) \phi(t,\d v, \d x) \right) \iint x f(v) R(z) [\partial_v \mu_\infty] \d v \d x.
\end{equation}
For a non-trivial perturbation, the scalar quantity $\iint x f(v) \phi(\d v, \d x)$ is non-zero. Dividing it out yields the self-consistency condition:
\begin{equation}
    \label{eq:self_consistency}
    1 + J \iint x f(v) R(z) [\partial_v \mu_\infty] \d v \d x = 0.
\end{equation}

We now connect this resolvent term back to the linear response function $\Theta_\alpha(t)$ defined in \eqref{eq:thetaintro}. Expanding the resolvent as a Laplace transform, we have:
\[
    J \iint x f(v) R(z) [\partial_v \mu_\infty] \d v \d x = \int_0^\infty e^{-z t} \left[ J \iint x f(v) \mathcal{T}_\alpha^*(t) [\partial_v \mu_\infty] \d v \d x \right] \d t.
\]
The inner integral is equal to:
\[
    J \iint \mathbb{E}_{(v,x)} \left[ X_t^\alpha f(V_t^\alpha) \right] \partial_v \mu_\infty(\d v, \d x).
\]
Integrating by parts with respect to $v$, we obtain
\[
    - J \iint \frac{\partial}{\partial v} \mathbb{E}_{(v,x)} \left[ X_t^\alpha f(V_t^\alpha) \right] \mu_\infty(\d v, \d x).
\]
By definition, this quantity is precisely $-\Theta_\alpha(t)$. Therefore, the integral term in our self-consistency condition \eqref{eq:self_consistency} is simply the negative Laplace transform of $\Theta_\alpha(t)$:
\[
    J \iint x f(v) R(z) [\partial_v \mu_\infty] \d v \d x = - \int_0^\infty e^{-z t} \Theta_\alpha(t) \d t = - \widehat{\Theta}_\alpha(z).
\]
Substituting this result back into \eqref{eq:self_consistency}, we finally arrive at the spectral condition for the linearized system. This confirms that the local stability of the invariant measure $\mu_\infty$ is governed by the roots of the equation $\widehat{\Theta}_\alpha(z) = 1$ in the complex plane, matching the criterion stated in our main theorem, Theorem~\ref{theo:2main}.

\section*{Acknowledgements} The authors warmly thank Romain Veltz to whom we owe the arguments of Section \ref{app:spectral_heuristic}. VS is supported by a fellowship from the Swiss National Science Foundation (grant no. 222150).

\bibliographystyle{abbrvnat}
\bibliography{biblio}
\end{document}